\documentclass[11pt,a4paper,reqno]{amsart}
\usepackage[left=3cm,right=3cm,top=3cm,bottom=3cm]{geometry}
\usepackage{mathtools}
\usepackage{amsthm}
\usepackage{amsfonts}
\usepackage{amssymb}
\usepackage{enumitem}
\setlist[enumerate]{label=\emph{(\roman*)}}

\newtheorem{theorem}{Theorem}
\newtheorem{lemma}{Lemma}
\newtheorem{proposition}{Proposition}
\newtheorem{claim}{Claim}
\theoremstyle{remark}
\newtheorem{remark}{Remark}

\newcommand\RR{\mathbb{R}}
\newcommand\loc{\rho}
\newcommand\YYp{\boldsymbol{Y}_{\hspace{-0.3em}+}}
\newcommand\ZZp{\boldsymbol{Z}_{\hspace{-0.1em}+}}
\newcommand\YYm{\boldsymbol{Y}_{\hspace{-0.3em}-}}

\newcommand\YYpm{\boldsymbol{Y}_{\hspace{-0.3em}\pm}}
\newcommand\ZZpm{\boldsymbol{Z}_{\hspace{-0.1em}\pm}}
\newcommand\ee{\boldsymbol{\varepsilon}}
\newcommand\uu{\boldsymbol{u}}
\newcommand\pp{\boldsymbol{\phi}}

\DeclareMathOperator\sech{sech}
\DeclareMathOperator\sgn{sgn}

\title[Soliton dynamics for 1D NLKG]{Soliton dynamics for the 1D NLKG equation with symmetry and in the absence of internal modes}

\author{Micha{\l} Kowalczyk}
\address{Departamento de Ingenier\'{\i}a Matem\'atica and Centro
de Modelamiento Matem\'atico (UMI 2807 CNRS), Universidad de Chile, Casilla
170 Correo 3, Santiago, Chile.}
\email {kowalczy@dim.uchile.cl}

\author{Yvan Martel}
\address{CMLS, \'Ecole polytechnique, CNRS, 91128 Palaiseau Cedex, France}
\email{yvan.martel@polytechnique.edu}

\author{Claudio Mu\~noz}
\address{CNRS and Departamento de Ingenier\'{\i}a Matem\'atica and Centro
de Modelamiento Matem\'a\-tico (UMI 2807 CNRS), Universidad de Chile, Casilla
170 Correo 3, Santiago, Chile.}
\email{cmunoz@dim.uchile.cl}

\thanks{M.K. was partially funded by Chilean research grants FONDECYT 1170164.
C.M. was partially funded by Chilean research grants FONDECYT 1150202.
M.K. and C.M. were partially funded by project France-Chile ECOS-Sud C18E06 and CMM Conicyt PIA AFB170001.
Part of this work was done while C.M. and M.K. were visiting the CMLS at \'Ecole Polytechnique, France.
Part of this work was done while C.M. was visiting the Departamento de Matem\'aticas Aplicadas de Granada, UGR, Spain.}


\begin{document}

\begin{abstract}
We consider the dynamics of even solutions of the one-dimensional nonlinear Klein-Gordon equation
$\partial_t^2 \phi - \partial_x^2 \phi + \phi - |\phi|^{2\alpha} \phi =0$ for $\alpha>1$,
in the vicinity of the unstable soliton $Q$.
Our main result is that stability in the energy space $H^1(\RR)\times L^2(\RR)$ implies asymptotics stability in a local energy norm.
In particular, there exists a Lipschitz graph of initial data leading to stable and asymptotically stable trajectories.

The condition $\alpha>1$ corresponds to cases where the linearized operator around $Q$ has no resonance and no internal mode.
Recall that the case $\alpha>2$ is treated in~\cite{KNS} using Strichartz and other local dispersive estimates. Since these tools are not 
 available for low power nonlinearities, our approach is based on virial type estimates and the particular structure of the linearized operator observed in~\cite{CGNT}.
\end{abstract}

\maketitle

\section{Introduction}
\subsection{Main results}
Consider the one-dimensional focusing nonlinear Klein-Gordon equation
\begin{equation}\label{nlkg}
\partial_t^2 \phi - \partial_x^2 \phi + \phi - f(\phi) =0 ,\quad (t,x)\in \RR\times \RR,
\quad f(\phi) = |\phi|^{2\alpha} \phi,
\end{equation}
where $\alpha>0$. This equation also rewrites as a first order system in time for the function $\pp=(\phi,\partial_t \phi)=(\phi_1,\phi_2)$,
\begin{equation*}
\left\{\begin{aligned}
& \dot \phi_1 = \phi_2 \\
& \dot \phi_2 = \partial_x^2 \phi_1 - \phi_1 + f(\phi_1) . 
\end{aligned}\right.\end{equation*}
Let $F(\phi)=\int_0^\phi f(s)ds=\frac{1}{2\alpha+2}|\phi|^{2\alpha+2}$. Note that~\eqref{nlkg} is Hamiltonian. The conservation of energy of a solution $(\phi,\partial_t\phi)$ of~\eqref{nlkg} writes
\begin{equation}\label{eq:energy}
E(\phi,\partial_t\phi)=
\frac 12 \int \left\{(\partial_t \phi)^2 + (\partial_x \phi)^2 + \phi^2 - 2 F(\phi)\right\}
=E(\phi(0),\partial_t\phi(0)).
\end{equation}
For initial data in the energy space $H^1\times L^2$, local well-posedness, as well as global well-posedness for small solutions, is well-known (see for example~\cite{CaHa}, Theorem~6.2.2 and Proposition~6.3.3).

\smallskip

Denote by $Q$ the standing wave solution of~\eqref{nlkg}, also called \emph{soliton}, explicitly given by
\begin{equation*}
Q(x)=\frac {(\alpha+1)^{\frac 1{2\alpha}}}{\cosh^{\frac 1\alpha}(\alpha x)},\qquad Q''-Q+Q^{2\alpha +1}=0
\quad \mbox{on $\RR$.}
\end{equation*}
The linearized operator $L$ around $Q$ writes
\begin{equation}\label{def:L}
L =-\partial_x^2 + 1 - (2\alpha+1) Q^{2\alpha}
= -\partial_x^2 + 1 - \frac {(2\alpha+1)(\alpha+1)}{\cosh^2(\alpha x)}.
\end{equation}
For any $\alpha>0$, the first eigenvalue of $L$ is $\lambda_0=-\alpha(\alpha+2)=-\nu_0^2$ ($\nu_0>0$) with corresponding normalized eigenfunction
\begin{equation}\label{eq:Y0}
Y_0(x)=c_0 \left(\cosh(\alpha x)\right)^{-(1+\frac 1\alpha)},\quad 
\langle Y_0,Y_0\rangle =1,\quad
LY_0=-\nu_0^2 Y_0
\end{equation}
(we denote $\langle A,B\rangle=\int A\cdot B$). The second eigenvalue of $L$ is $0$ with eigenfunction $Y_1=c_1Q'$.
In the case $\alpha>1$, there is no other eigenvalue in $[0,1)$, which means that there is no \emph{internal mode}
for the model
(see Section~\ref{S.1.3}).

Let
\begin{equation*}
\YYpm=\left(\begin{array}{c}Y_0 \\ \pm \nu_0 Y_0\end{array}\right),\quad
\ZZpm=\left(\begin{array}{c} Y_0 \\ \pm \nu_0^{-1} Y_0\end{array}\right).
\end{equation*}
The functions $\uu_{\pm}(t,x)=e^{\pm \nu_0 t} \YYpm(x)$ are solutions of the linearized problem
\begin{equation}\label{lin}\left\{\begin{aligned}
& \dot u_1 = u_2 \\
& \dot u_2 = - L u_2 
\end{aligned}\right.\end{equation}
illustrating the presence of exponentially stable and unstable modes both relevant in the dynamics of solutions in the vicinity of a soliton.

\smallskip

In this paper, by \emph{global solution of}~\eqref{nlkg}, we mean a function $\pp\in \mathcal C([0,\infty),H^1\times L^2)$ satisfying
\eqref{nlkg} for all $t\geq 0$. We only consider solutions with even symmetry.

\smallskip

Our main result is the following \emph{conditional asymptotic stability theorem}.
\begin{theorem}\label{th:1}
Let $\alpha>1$. There exists $\delta>0$ such that if a global even solution $\pp=(\phi,\partial_t \phi)$ of~\eqref{nlkg} satisfies
\begin{equation}\label{th:1:stab}
\mbox{for all $t\geq 0$,}\quad \|\pp(t)-(Q,0)\|_{H^1(\RR)\times L^2(\RR)}< \delta
\end{equation}
then, for any interval $I$ of $\RR$,
\begin{equation}\label{th:1:asymp}
\lim_{t\to +\infty} \|\pp(t)- (Q,0)\|_{H^1(I)\times L^2(I)}=0.
\end{equation}
\end{theorem}

For the sake of completeness, we provide a description of the set of initial data leading to global solutions
satisfying the stability assumption~\eqref{th:1:stab} (see also Theorem~4.1 in~\cite{BJ}).

For $\delta_0>0$, let
\begin{equation}\label{def:A_delta}
\mathcal A_0=
\left\{\ee\in H^1(\RR)\times L^2(\RR) \mbox{ such that $\ee$ is even, } \|\ee\,\|_{H^1\times L^2}< \delta_0 \mbox{ and }
\langle \ee , \ZZp\rangle=0\right\}.
\end{equation}

\begin{theorem}\label{th:2}
Let $\alpha>1$.
There exist $C, \delta_0>0$ and a Lipschitz function
$
h\colon\mathcal A_0 \to \RR
$
with $
h(0)=0$ and $|h(\ee)|\leq C \|\ee\|_{H^1\times L^2}^{3/2}$
 such that denoting
\begin{equation*}
 \mathcal M=\left\{(Q,0)+\ee+h(\ee) \YYp \mbox{ with } \ee\in \mathcal A_0\right\}
\end{equation*}
the following holds
\begin{enumerate}
\item If $\pp_0\in \mathcal M$ then the solution $\pp$ of~\eqref{nlkg} with initial data $\pp_0$ 
is global and satisfies, for all $t\geq 0$,
\begin{equation}\label{prop:1:stab}
 \|\pp(t)-(Q,0)\|_{H^1(\RR)\times L^2(\RR)}\leq C\|\pp_0-(Q,0)\|_{H^1(\RR)\times L^2(\RR)}.
\end{equation}
\item If a global even solution $\pp$ of~\eqref{nlkg} satisfies, for all $t\geq 0$,
\begin{equation*}
\|\pp(t)-(Q,0)\|_{H^1(\RR)\times L^2(\RR)}< \tfrac 12 \delta_0,
\end{equation*}
 then for all $t\geq 0$, $\pp(t)\in \mathcal M$.
\end{enumerate}
\end{theorem}

\subsection{Related results and comments on the proof}
First, we comment on two articles devoted to soliton dynamics for the one-dimensional nonlinear Klein-Gordon equation~\eqref{nlkg}.

Using techniques
based on Strichartz and other local dispersive estimates, 
Krieger et al.~\cite{KNS} have completely treated the case $\alpha >2$ in the case of even data.
Indeed, they classify all solutions whose energy does not exceed too much that of the ground state~$Q$.
This includes the construction, by the fixed point argument,  of 
a $\mathcal C^1$ center-stable manifold around the soliton  and the proof  of  asymptotic stability and scattering (linear behavior) around the ground state for solutions on the manifold. The method seems limited to $\alpha\geq 2$ because of the use of Strichartz estimates to control the nonlinear term, see comment in Section 3.4 of~\cite{KNS}.

By formal and numerical methods, Bizo\'n et al.~\cite{BCS} have shown that for even solutions trapped by the soliton,
the convergence rate to $Q$ heavily depends on the power $\alpha$ of the nonlinearity.
In the $L^\infty$ sense, they conjecture the following trichotomy:
(a) fast dispersive decay for $\alpha >1$; (b) slow decay for $\alpha=1$; (c) very slow decay for $0<\alpha<1$.
The threshold value $\alpha=1$ corresponds to the emergence of a resonance at the linear level, while $\alpha<1$ leads to 
one or several internal modes (see Section~\ref{S.1.3}).
Following these observations, unifying the case $\alpha>1$ was the main motivation of the present work.

Our method does not give an explicit decay rate as $t\to +\infty$, but 
we notice as a by-product of the proof of Theorem~\ref{th:1} that, for any interval $I$ of $\RR$, it holds
\begin{equation}\label{convergence}
\int_0^{+\infty} \|\pp(t)-(Q,0)\|_{H^1(I)\times L^2(I)}^2 dt <\infty.
\end{equation}
This is to be compared with the results obtained in~\cite{KMM} on the (local) asymptotic stability of the kink for the $\phi^4$ model 
under small odd perturbations. Indeed, in the latter case, the presence of an internal mode leads to a lower convergence rate since the component $z(t)$ of the solution along the internal mode only satisfies the weaker
estimate $\int_0^{+\infty} |z(t)|^4 dt <\infty$
(see Theorem~1.2 in~\cite{KMM}). 
Although we do not claim optimality of such results, in the case of~\eqref{nlkg} with $0<\alpha\leq 1$, we do not expect estimates such as in~\eqref{convergence} to hold.

The proof of Theorem~\ref{th:1} is mainly based on localized virial type arguments similar to that used in~\cite{KMM,MM1,MaMeRa}, for example. Unlike in these works, we avoid numerical computations of certain constants related to the coercivity of the virial functional by using factorization properties of the linearized operator described in~\cite{CGNT} (see also references~\cite{MS,SuSu}, cited in~\cite{CGNT}).
A formal presentation of this approach is given in~Section~\ref{S:heuristic}.
We point out that the same structure was crucially used in the construction of blow-up solutions for the wave maps, Yang-Mills
and $O(3)$ $\sigma$-models in~\cite{RR,RS}. Note that in the present paper, we compensate the loss of two derivatives due to the change of variables to still work in the energy space.

We refer to~\cite{BC,KK1,KK2,KoMaMu1,KoMaMu2,LS,Ste,SW} for various results of asymptotic stability for the nonlinear Klein-Gordon equation and $\phi^4$ equation or variants of these models.

Several other conditional asymptotic stability results or classifications in a neighborhood of the ground state for the nonlinear Klein-Gordon in higher dimensions and for the nonlinear Schr\"odinger equation were also obtained in~\cite{Cc,CPe,NS1,Sc09}, for example.
We also mention~\cite{KS} where for the mass supercritical Schr\"odinger equation in one dimension, a finite co-dimensional manifold of initial data trapped by the soliton was constructed.

Concerning the generalized Korteweg-de Vries equation and related models, studies of the dynamics of the solutions close to the soliton are presented in~\cite{CMPS,GS2,KM,Ma,MaMe,MaMeRa,MaMeNaRa,PW2}, in blow-up contexts or for bounded solutions. Note that the method introduced in~\cite{Ma,MaMe}, using the special structure of a transformed linearized problem, also has some analogy with our proof.

For global existence results in the case of semilinear and quasilinear wave equations, we refer to~\cite{Delort,DeF}.

Finally,  we refer to~\cite{BJ,BLZ} and references therein for refined descriptions of dynamics of solutions in various settings.

\subsection{Resonances and internal modes}\label{S.1.3}
As mentioned before, the absence of any other eigenvalue in $[0,1)$ for the operator $L$ when $\alpha>1$ is important in our proof.
For $0<\alpha\leq 1$, we continue the description of the spectrum of $L$.
For $\alpha=1$, there is an even resonance at $1$. For any $0<\alpha<1$, there is a third eigenvalue associated to an even eigenfunction
\begin{equation*}
Y_2(x)=c_2Y_0(x) \left( 1-\frac 2\alpha \sinh^2(\alpha x)\right),\quad \lambda_2=\alpha(2-\alpha),\quad
\nu_2=\lambda_2^{\frac 12}.
\end{equation*}
In particular, for any $0<\alpha<1$, the function
\begin{equation*}
\uu(t)= (\cos(\nu_2 t ) Y_2,-\nu_2 \sin(\gamma_2 t ) Y_2)
\end{equation*}
is solution of~\eqref{lin}.
These solutions are typical of the notion of \emph{internal modes} and show that asymptotic stability (even up to the exponential instable mode) cannot be true at the linear level for such value of $\alpha$.
An important issue is the nature of the interaction of such internal mode with the nonlinearity.
We recall that such an internal mode was treated in the context of the $\phi^4$ equation in ~\cite{KMM}.
Pioneering results on internal modes were obtained in~\cite{SW}. See other references in~\cite{KMM}.

For $\alpha\in (\frac 12,1)$, there are no other eigenvalue on $[0,1)$. 
For $\alpha=\frac 12$, there is an odd resonance at $1$. 
For $\alpha\in (\frac 13,\frac 12)$, there is a fourth eigenvalue, associated to an odd eigenfunction.
For $\alpha\in (\frac 14,\frac 13)$, there are five eigenvalues, three of them being associated to even eigenfunctions.
In particular, there are two even internal modes. This procedure can be continued for all $\alpha>0$, showing the emergence of arbitrarily many internal modes (and sometimes resonances) as $\alpha\to 0^+$. 

The above information is taken from Section~3 of~\cite{CGNT}.

\section{Preliminaries}
\subsection{Decomposition of a solution in a vicinity of the soliton}\label{S:2}
Let $\pp=(\phi,\partial_t \phi)$ be a solution of~\eqref{nlkg} satisfying~\eqref{th:1:stab} for some small $\delta>0$.
We decompose $(\phi,\partial_t \phi)$ as follows
\begin{equation}\label{eq:decomp}\left\{\begin{aligned}
&\phi(t,x)=Q(x)+a_1(t)Y_0(x)+u_1(t,x)\\
&\partial_t \phi(t,x)=a_2(t)\nu_0 Y_0(x)+u_2(t,x)
\end{aligned}\right.\end{equation}
where
\begin{equation*}
a_1(t)=\langle \phi(t) -Q,Y_0\rangle,\quad
a_2(t)=\frac1{\nu_0}\langle \partial_t\phi(t),Y_0\rangle,
\end{equation*}
so that
\begin{equation}\label{myortho}
\langle u_1(t),Y_0\rangle=\langle u_2(t),Y_0\rangle=0.
\end{equation}
Setting
\begin{equation}\label{def:b}
b_+=\frac 12 (a_1+a_2),\quad ~b_-=\frac 12(a_1-a_2),
\end{equation}
we observe that $\pp$ also writes as
\begin{equation}\label{decompb}
\pp = (Q,0)+\uu +b_- \YYm+ b_+\YYp,\quad \uu=(u_1,u_2).
\end{equation}
From~\eqref{th:1:stab}, for all $t\in [0,\infty)$, it holds
\begin{equation}\label{small}
 \|u_1(t)\|_{H^1}+\|u_2(t)\|_{L^2}+|a_1(t)|+|a_2(t)|+|b_+(t)|+|b_-(t)|\leq C_0 \delta.
\end{equation}
Moreover, using $Q''-Q+f(Q)=0$, $LY_0=-\nu_0^2 Y_0$ and~\eqref{myortho}, the systems of equations of~$(a_1,a_2)$ and~$(u_1,u_2)$ write
\begin{equation}\label{eq:a}\left\{
\begin{aligned}
& \dot a_1 = \nu_0 a_2 \\
& \dot a_2 = \nu_0 a_1 + \frac{N_0}{\nu_0}
\end{aligned}
\right.
\qquad \mbox{equivalently}\qquad
\left\{
\begin{aligned}
& \dot b_+ = \nu_0 b_+ + \frac{N_0}{2\nu_0} \\
& \dot b_- = -\nu_0 b_- - \frac{N_0}{2\nu_0} 
\end{aligned}
\right.
\end{equation}
and
\begin{equation}\label{eq:u}\left\{
\begin{aligned}
& \dot u_1 = u_2 \\
& \dot u_2 =- L u_1 + N^\perp
\end{aligned}
\right.\end{equation}
where
\begin{equation}\label{def:N}\begin{aligned}
&N=f(Q+a_1Y_0+u_1)-f(Q)-f'(Q)a_1Y_0-f'(Q)u_1, 
\\&N_0=\langle N,Y_0\rangle,\quad N^\perp= N- N_0 Y_0.
\end{aligned}\end{equation}

\subsection{Notation for virial arguments}
Let $\rho$ be the following weight function 
\begin{equation}\label{def:rho}
\rho(x)=\sech\left( \frac{x}{10}\right).
\end{equation}
For any function $w\in H^1$, consider the norm
\begin{equation}\label{def:w_loc}
\|w\|_\loc=\left[ \int \left((\partial_x w)^2 + \rho w^2 \right)\right]^{\frac 12}.
\end{equation}
We consider a smooth even function $\chi:\RR\to \RR$ satisfying
\begin{equation}\label{on:chi}
\mbox{$\chi=1$ on $[-1,1]$, $\chi=0$ on $(-\infty,-2]\cup[2,+\infty)$,
$\chi'\leq 0$ on $[0,+\infty)$.}
\end{equation}
For $A>0$, we define the functions $\zeta_A$ and $\varphi_A$ as follows
\begin{equation*}
\zeta_A(x)=\exp\left(-\frac1A(1-\chi(x))|x|\right),\quad
\varphi_A(x)=\int_0^x \zeta_A^2(y) dy,\quad x\in \RR.
\end{equation*}
For $B>0$, we also define
\begin{equation}\label{def:zetaB}
\zeta_B(x)=\exp\left(-\frac1B(1-\chi(x))|x|\right),\quad
\varphi_B(x)=\int_0^x \zeta_B^2(y) dy,\quad x\in \RR.
\end{equation}
and we consider the function $\psi$ defined as
\begin{equation}\label{def:chiB}
\psi_B(x)=\chi_B^2(x)\varphi_B(x)\quad\mbox{where}\quad
\chi_B(x)=\chi\left(\frac{x}{B^2}\right),\quad x\in \RR.
\end{equation}
The notation $X\lesssim Y$ means $X\leq CY$ for a constant independent of $A$ and $B$.

\smallskip

These functions $\zeta_A$, $\varphi_A$, $\zeta_B$, $\varphi_B$ and $\psi_B$ will be used in two distinct virial arguments with different scales 
\begin{equation}\label{eq:A_B}
A\gg B^2\gg B\gg 1.
\end{equation}

\section{Virial argument in $u$}
Set
\begin{equation}\label{def:I}
\mathcal I= \int \left(\varphi_A\partial_x u_1+\frac 12 \varphi_A' u_1\right)u_2,
\end{equation}
and
\begin{equation}\label{def:w}
w=\zeta_A u_1.
\end{equation}
We refer to~\cite{KMM} for the use of such virial argument in a similar context.
Here, $w$ represents a localized version of $u_1$, in the scale $A$ (see~\eqref{eq:A_B}). We shall prove the following result.
\begin{proposition}\label{pr:virialu}
There exist $C_1>0$ and $\delta_1>0$ such that for any $0<\delta\leq \delta_1$, the following holds.
Fix $A=\delta^{-1}$. Assume that for all $t\geq 0$,~\eqref{small} holds.
Then, for all $t\geq 0$,
\begin{equation}\label{eq:pr:u}
\dot{\mathcal I}\leq -\frac 12 \int (\partial_x w)^2 + C_1\int \sech\left(\frac x2\right) w^2+C_1 |a_1|^4.
\end{equation}
\end{proposition}

\begin{remark}
Note that estimate~\eqref{eq:pr:u} does not involve any type of spectral analysis. Its purpose is to give a simple control of $ \int (\partial_x w)^2 $ in terms of $\int \sech (\frac x2 ) w^2$ and $|a_1|^4$.
\end{remark}

The rest of this section is devoted to the proof of Proposition~\ref{pr:virialu}.
We compute from~\eqref{def:I}
\begin{equation*}
\dot{\mathcal I} = \int \left(\varphi_A\partial_x \dot u_1+\frac 12 \varphi_A' \dot u_1\right)u_2
+\int \left(\varphi_A\partial_x u_1+\frac 12 \varphi_A' u_1\right)\dot u_2.
\end{equation*}
Replacing $\dot u_1$ by $u_2$ and integrating by parts, the first integral in the right-hand side vanishes.
The expression of $\dot u_2$ in~\eqref{eq:u} rewrites
\begin{equation*}
\dot u_2 = \partial_x^2 u_1-u_1+f(Q+a_1Y_0+u_1)-f(Q)-f'(Q)a_1Y_0-N_0Y_0,
\end{equation*}
and so
\begin{align*}
\dot{\mathcal I} &= 
\int \left(\varphi_A\partial_x u_1+\frac 12 \varphi_A' u_1\right) \left(\partial_x^2 u_1-u_1\right)
\\&
\quad +\int \left(\varphi_A\partial_x u_1+\frac 12 \varphi_A' u_1\right)\left[f(Q+a_1Y_0+u_1)-f(Q)-f'(Q)a_1Y_0-N_0Y_0\right].
\end{align*}
To treat the first line in the expression of $\dot{\mathcal I}$, we claim the following. 
\begin{lemma}\label{le:1}
It holds
\begin{equation}\label{eq:1}
\int \left(\varphi_A\partial_x u_1+\frac 12 \varphi_A' u_1\right)(\partial_x^2u_1-u_1)
= - \int (\partial_x w)^2 - \frac 12 \int \left(\frac{\zeta_A''}{\zeta_A}-\frac{(\zeta_A')^2}{\zeta_A^2}\right) w^2.
\end{equation}
Moreover
\begin{equation}\label{sur:zetaA}
\frac{\zeta_A''}{\zeta_A}-\frac{(\zeta_A')^2}{\zeta_A^2}
=\frac 1A\left[\chi''(x) |x|+2\chi'(x)\sgn(x)\right]
\end{equation}
and
\begin{equation}\label{bd:A}
\left|\frac{\zeta_A''}{\zeta_A}-\frac{(\zeta_A')^2}{\zeta_A^2}\right|\lesssim \frac {\mathbf{1}_{1\leq |x|\leq 2} (x)}A\lesssim \frac{\sech(x)}A.
\end{equation}
\end{lemma}
\begin{proof}
Proof of~\eqref{eq:1}.
By integration by parts
\begin{equation*}
\int \left(\varphi_A\partial_x u_1+\frac 12 \varphi_A' u_1\right)(\partial_x^2u_1-u_1)
=-\int \varphi_A' (\partial_x u_1)^2+\frac 14 \int \varphi_A'''u_1^2.
\end{equation*}
We rewrite the above expression using the auxiliary function $w$.
Indeed,
\begin{align*}
\int (\partial_x w)^2 & = \int (\zeta_A \partial_x u_1+\zeta_A'u_1)^2 
= \int \zeta_A^2(\partial_x u_1)^2 + 2 \int \zeta_A\zeta_A' u_1 \partial_x u_1 + \int (\zeta_A')^2 u_1^2\\
& = \int \varphi_A' (\partial_x u_1)^2 - \int \zeta_A\zeta_A'' u_1^2
=\int \varphi_A' (\partial_x u_1)^2-\int \frac{\zeta_A''}{\zeta_A} w^2
\end{align*}
and so
\begin{equation*}
\int \varphi_A' (\partial_x u_1)^2=\int (\partial_x w)^2+\int \frac{\zeta_A''}{\zeta_A} w^2.
\end{equation*}
Next,
\begin{equation}\label{avoir}
\int \varphi_A''' u_1^2 = \int \frac{(\zeta_A^2)''}{\zeta_A^2} w^2 
=2 \int \left(\frac{\zeta_A''}{\zeta_A}+\frac{(\zeta_A')^2}{\zeta_A^2}\right) w^2.
\end{equation}
Identity~\eqref{eq:1} follows.
\smallskip

Proof of~\eqref{sur:zetaA}-\eqref{bd:A}.
By elementary computations, we have
\begin{equation*}
\frac{\zeta_A'}{\zeta_A}=-\frac 1A \left[ -\chi'(x) |x|+(1-\chi(x)) \sgn(x)\right], 
\end{equation*}
\begin{equation*}
\frac{\zeta_A''}{\zeta_A}=\frac 1{A^2} \left[ -\chi'(x) |x|+(1-\chi(x)) \sgn(x)\right]^2
+\frac 1A \left[\chi''(x) |x| +2\chi'(x) \sgn(x)\right],
\end{equation*}
which proves~\eqref{sur:zetaA}.
Estimate~\eqref{bd:A} then follows from the definition of $\chi$.
\end{proof}

To treat the second line in the expression of $\dot{\mathcal I}$, we claim the following.
\begin{lemma}\label{le:2}
\begin{equation}\label{eq:2}\begin{aligned}
&\left|\int \left(\varphi_A\partial_x u_1+\frac 12 \varphi_A' u_1\right)
\left[f(Q+a_1Y_0+u_1)-f(Q)-a_1 f'(Q)Y_0-N_0Y_0\right]\right|
\\
&\quad \lesssim |a_1|^4+ \int \sech\left( \frac x2\right) w_1^2+
A^2\|u_1\|_{L^\infty}^{2\alpha} \int |\partial_x w|^2.
\end{aligned}\end{equation}
\end{lemma}
\begin{proof}
First, we treat the term $-\int \left(\varphi_A\partial_x u_1+\frac 12 \varphi_A' u_1\right)N_0Y_0$.
By Taylor's expansion, one has
\begin{equation}\label{sur:N}
|N|\lesssim a_1^2 Q^{2\alpha-1} Y_0^2 + Q^{2\alpha-1} u_1^2
+ |a_1|^{2\alpha+1} Y_0^{2\alpha+1}+|u_1|^{2\alpha+1},
\end{equation}
and thus, by decay estimates on $Q$ and $Y_0$, and by~\eqref{small}, $|a_1|\lesssim 1$, $\|u_1\|_{L^\infty}\lesssim
\|u_1\|_{H^1}\lesssim 1$, $A\geq 4$, it holds
\begin{equation}\label{sur:N0}
|N_0|\lesssim a_1^2 + \int \sech\left(x\right) u_1^2
\lesssim a_1^2 + \int \sech\left(\frac x2\right) w^2.
\end{equation}
Using integration by parts, 
\begin{equation*}
-\int \left(\varphi_A\partial_x u_1+\frac 12 \varphi_A' u_1\right) Y_0
= \int u_1 \left(\varphi_A\partial_x Y_0+\frac 12 \varphi_A' Y_0\right).
\end{equation*}
Note that for all $x\in \RR$, $|\varphi_A'(x)|\leq 1$ and $|\varphi_A(x)|\leq |x|$, and so
\begin{equation}\label{on:varphiA}
|\varphi_A(x) \sech(x)|+|\varphi_A'(x) \sech(x)|\leq (|x|+1) \sech(x)\lesssim \sech\left( \frac 34 x\right),
\end{equation}
for an implicit constant independent of $A$.
Thus, by the Cauchy-Schwarz inequality,
\begin{align*}
\left|N_0 \int \left(\varphi_A\partial_x u_1+\frac 12 \varphi_A' u_1\right)Y_0 \right|
\lesssim a_1^4 + \int \sech\left(\frac x2\right) w_1^2.
\end{align*}

Second, we decompose
\begin{align*}
&\int \left(\varphi_A\partial_x u_1+\frac 12 \varphi_A' u_1\right)\left[f(Q+a_1Y_0+u_1)-f(Q)-f'(Q)a_1Y_0\right]\\
&\quad = \int \varphi_A \partial_x \left[ F(Q+a_1Y_0+u_1)-F(Q+a_1Y_0)-(f(Q)+f'(Q)a_1Y_0)u_1\right]\\
&\qquad -\int \varphi_A Q'\left[f(Q+a_1Y_0+u_1)-f(Q+a_1Y_0)-(f'(Q)+f''(Q)a_1Y_0)u_1\right]\\
&\qquad -a_1 \int \varphi_A Y_0' \left[f(Q+a_1Y_0+u_1)-f(Q+a_1Y_0)-f'(Q)u_1\right]\\
&\qquad + \frac 12 \int \varphi_A' u_1\left[f(Q+a_1Y_0+u_1)-f(Q)-f'(Q)a_1Y_0\right]\\
&\quad=I_1+I_2+I_3+I_4.
\end{align*}
We rewrite $I_1$, $I_2$, $I_3$ and $I_4$ as follows
\begin{align*}
I_1&=-\int \varphi_A'\left[ F(Q+a_1Y_0+u_1)-F(Q+a_1Y_0)-F'(Q+a_1Y_0)u_1-F(u_1)\right]\\
&\quad - \int \varphi_A'\left[f(Q+a_1Y_0)-f(Q)-f'(Q)a_1Y_0\right]u_1-\int \varphi_A' F(u_1),
\end{align*}
\begin{align*}
I_2&=-\int \varphi_A Q'\left[f(Q+a_1Y_0+u_1)-f(Q+a_1Y_0)-f'(Q+a_1Y_0)u_1\right]\\
&\quad -\int \varphi_A Q'\left[f'(Q+a_1Y_0)-f'(Q)-f''(Q)a_1Y_0\right]u_1,
\end{align*}
\begin{align*}
I_3&= -a_1 \int \varphi_A Y_0' \left[f(Q+a_1Y_0+u_1)-f(Q+a_1Y_0)-f'(Q+a_1Y_0)u_1\right]\\
&\quad-a_1\int \varphi_A Y_0' \left[f'(Q+a_1Y_0)-f'(Q)\right]u_1,
\end{align*}
and
\begin{align*}
I_4&=\frac 12 \int \varphi_A' u_1\left[f(Q+a_1Y_0+u_1)-f(Q+a_1Y_0)-f(u_1)\right]\\
&\quad 
+\frac 12 \int \varphi_A' u_1 [f(Q+a_1Y_0)-f(Q)-f'(Q)a_1Y_0]
+\frac 12 \int \varphi_A' u_1f(u_1).
\end{align*}

To control the two terms that are purely nonlinear in $u_1$, we need the following claim.
\begin{claim}\label{cl:nonlin}
It holds
\begin{equation}\label{eq:3}
\int \zeta_A^2 |u_1|^{2\alpha+2}
=\int \zeta_A^{-2\alpha} |w|^{2\alpha+2}
\lesssim A^2\|u_1\|_{L^\infty}^{2\alpha} \int |\partial_x w|^2.
\end{equation}
\end{claim}
\begin{proof}[Proof of Claim~\ref{cl:nonlin}]
The first equality in~\eqref{eq:3} corresponds to the definition of $w$ in~\eqref{def:w}.
Next, by integration by parts and standard estimates, we have
\begin{align*}
&\int_0^{+\infty} \exp\left(\frac{2\alpha}{A}x\right) |w|^{2\alpha+2} dx\\
&=-\frac{A}{2\alpha} |w(0)|^{2\alpha+2}-\frac{A}{2\alpha} \int_0^{+\infty}\exp\left(\frac{2\alpha}{A}x\right)\partial_x\left( |w|^{2\alpha+2}\right)dx
\\ 
&\quad \leq -\frac{\alpha+1}{\alpha}A \int_0^{+\infty}\exp\left(\frac{2\alpha}{A}x\right)(\partial_x w)w |w|^{2\alpha}dx \\
&\quad\leq \frac{\alpha+1}{\alpha}A\|u_1\|_{L^\infty}^{\alpha} \int_0^{+\infty}\exp\left(\frac{\alpha}{A}x\right)|\partial_x w||w|^{\alpha+1}dx\\ 
&\quad \leq \left(\frac{\alpha+1}{\alpha}\right)^2A^2 \|u_1\|_{L^\infty}^{2\alpha} \int_0^{+\infty} |\partial_x w|^2dx
+\frac 14 \int_0^{+\infty} \exp\left(\frac{2\alpha}{A}x\right) |w|^{2\alpha+2}dx.
\end{align*}
Thus,
\begin{equation*}
\int_0^{+\infty} \exp\left(\frac{2\alpha}{A}x\right) |w|^{2\alpha+2}dx\leq 
\frac 43\left(\frac{\alpha+1}{\alpha}\right)^2A^2 \|u_1\|_{L^\infty}^{2\alpha} \int_0^{+\infty} |\partial_x w|^2dx,
\end{equation*}
which implies~\eqref{eq:3}.
\end{proof}
In particular,~\eqref{eq:3} implies that
\begin{equation*}
\int \varphi_A' F(u_1) + \int \varphi_A' u_1f(u_1)
\lesssim \int \zeta_A^2 |u_1|^{2\alpha+2}
\lesssim A^2\|u_1\|_{L^\infty}^{2\alpha} \int |\partial_x w|^2,
\end{equation*}
which takes care of the last terms in $I_1$ and $I_4$.

\smallskip

By Taylor expansion, $\alpha\geq 1$, $|a_1|\lesssim 1$ and $\|u_1\|_{L^\infty}\lesssim 1$, we have
\begin{align*}
&\left|F(Q+a_1Y_0+u_1)-F(Q+a_1Y_0)-F'(Q+a_1Y_0)u_1-F(u_1)\right|\\
&\quad \lesssim 
|Q+a_1Y_0|^{2\alpha}u_1^2+|Q+a_1Y_0| |u_1|^{2\alpha+1} \lesssim
\sech (x ) u_1^2 \lesssim 
\sech\left( \frac x2\right) w_1^2.
\end{align*}
Similarly, using also~\eqref{on:varphiA} and $A\geq 4$, we find the following estimates
\begin{equation*}
 \left| \varphi_A Q'\left[f(Q+a_1Y_0+u_1)-f(Q+a_1Y_0)-f'(Q+a_1Y_0)u_1\right]\right| \lesssim
\sech\left( \frac x2\right) w_1^2,
\end{equation*}
\begin{align*}
\left|a_1 \varphi_A Y_0' \left[f(Q+a_1Y_0+u_1)-f(Q+a_1Y_0)-f'(Q+a_1Y_0)u_1\right]\right|
\lesssim \sech\left( \frac x2\right) w_1^2,
\end{align*}
and
\begin{align*}
\left| \varphi_A' u_1\left[f(Q+a_1Y_0+u_1)-f(Q+a_1Y_0)-f(u_1)\right]\right|
\lesssim \sech\left( \frac x2\right) w_1^2.
\end{align*}
Moreover, again by Taylor expansion and \eqref{on:varphiA} (with $A>8$), we have
\begin{align*}
&\left| \varphi_A' \left[f(Q+a_1Y_0+u_1)-f(Q)-f'(Q)a_1Y_0\right]u_1\right|\\
& +\left| \varphi_A Q'\left[f'(Q+a_1Y_0)-f'(Q)-f''(Q)a_1Y_0\right]u_1 \right|\\
& +\left| a_1 \varphi_A Y_0' \left[f'(Q+a_1Y_0)-f'(Q)\right]u_1\right|\\
& +\left| \varphi_A' u_1 [f(Q+a_1Y_0)-f(Q)-f'(Q)a_1Y_0] \right|\\
&\quad\lesssim \sech\left( \frac x2\right) |a_1|^2 |u_1|
\lesssim \sech\left( \frac x2\right) w_1^2 + \sech\left( \frac x4\right) |a_1|^4.
\end{align*}
Collecting these estimates,~\eqref{eq:2} is proved.

Taking $\|u_1\|_{L^\infty}\leq \delta_A$, for $\delta_A$ small enough, we have proved
\begin{equation*}
\dot{\mathcal I}\leq - \int (\partial_x w)^2+C\int w^2 \sech\left(\frac x2\right)
+C a_1^4+A^2\|u_1\|_{L^\infty}^{2\alpha} \int (\partial_x w)^2.
\end{equation*}
Using $A=\delta^{-1}$ and $\|u_1\|_{L^\infty}^{2\alpha}\lesssim \delta^{2\alpha}$ (from \eqref{small}),
for $\delta_1$ small enough, we obtain~\eqref{eq:pr:u}.
\end{proof}
 
\section{Virial argument for the transformed problem}
\subsection{Heuristic}\label{S:heuristic}
We recall results from~\cite{CGNT}, pages 1086-1087.
Let
\begin{equation*}
L = -\partial_x^2 + 1 -(2\alpha+1)Q^{2\alpha},\quad
L_- = -\partial_x^2+1-Q^{2\alpha},
\end{equation*}
and
\begin{equation*}
U= Y_0 \cdot \partial_x \cdot Y_0^{-1},\quad
U^\star=-Y_0^{-1} \cdot \partial_x \cdot Y_0.
\end{equation*}
(The above notation means $Uf=Y_0( Y_0^{-1} f)'$.)
Then, the operators $L$ and $L_-$ rewrite as
$L=U^\star U+\lambda_0$, $L_-=U U^\star+\lambda_0$ and it follows that 
\begin{equation*}
U L = L_- U.
\end{equation*}
Now, let
\begin{equation}\label{op:L0}
L_0=-\partial_x^2+1+\frac {\alpha-1}{\alpha+1}Q^{2\alpha},
\end{equation}
and
\begin{equation*}
S= Q \cdot \partial_x \cdot Q^{-1},\quad
S^\star=-Q^{-1} \cdot \partial_x \cdot Q.
\end{equation*}
A similar structure $L_-=S^\star S$, $L_0=S S^\star$, leads to
\begin{equation*}
SL_-= L_0 S \quad \mbox{and thus}\quad 
SUL=L_0 SU.
\end{equation*}

In particular, let $(u_1,u_2)$ be a solution of~\eqref{lin}, and set $\tilde u_1=U u_1$, $\tilde u_2= U u_2$.
Then,
\begin{equation*}\left\{\begin{aligned}
& \dot {\tilde u}_1 = \tilde u_2 \\
& \dot {\tilde u}_2 = - L_- \tilde u_1.
\end{aligned}\right.\end{equation*}
Next, set 
\begin{equation*}v_1=S \tilde u_1=SU u_1 \quad \mbox{and}\quad v_2=S \tilde u_2=SU u_2.\end{equation*}
Then, $(v_1,v_2)$ satisfies the following transformed problem:
\begin{equation*}\left\{\begin{aligned}
& \dot v_1 = v_2 \\
& \dot v_2 = - L_0 v_1.
\end{aligned}\right.\end{equation*}
The key point for our analysis is that for $\alpha >1$, the potential in $L_0$ is positive.
This property happens to be the only spectral information needed for the proof of Theorem~\ref{th:1}.

Observe that $UY_0=0$, $U Q' =-\alpha Q$ and $SQ=0$, which means that the prior decomposition of the solution 
$(\phi,\partial_t \phi)$ as in Section~\ref{S:2}
and a coercivity argument as in Section~\ref{S:5} are necessary to avoid loosing information through the transformation.
(Here, we work with even functions and so only the direction $Y_0$ is relevant.)

\subsection{Transformed problem}
With respect to the above heuristic, we need to localize and regularize the functions involved.
For $\gamma>0$ small to be defined later, set
\begin{equation}\label{def:v1v2}
\left\{
\begin{aligned}
v_1&=(1-\gamma \partial_x^2)^{-1} SU (\chi_B u_1),\\
v_2&=(1-\gamma \partial_x^2)^{-1} SU (\chi_B u_2),
\end{aligned}
\right.
\end{equation}
where $\chi_B$ is defined in~\eqref{def:chiB}.
We refer to Section~\ref{S:5} for coercivity results relating $u_1$ and~$v_1$.
The introduction of the operator $(1-\gamma \partial_x^2)^{-1}$ with a small constant $\gamma$ is needed to compensate the loss of two derivatives due to the operator $SU$, without destroying the special algebra described heuristically. 
Now, we explain the role of the localization term $\chi_B$ in the definitions of $v_1$ and $v_2$.
Note that Proposition~\ref{pr:virialu} provides an estimate on the function $w$, which is a localized version of $u$
(see~\eqref{def:w}). To use this information, the functions $v_1$ and $v_2$ also need to contain a certain localization. 

We deduce the following system for $(v_1,v_2)$ from the one for $(u_1,u_2)$ in~\eqref{eq:u}
\begin{equation*}\left\{
\begin{aligned}
& \dot v_1 = v_2 \\
& \dot v_2 =- (1-\gamma \partial_x^2)^{-1} SU (\chi_B L u_1) + (1-\gamma \partial_x^2)^{-1} SU (\chi_B N^\perp ).
\end{aligned}
\right.\end{equation*}
First, we note that
\begin{equation*}
\chi_B L u_1 = L(\chi_B u_1) + 2 \chi_B' \partial_x u_1 +\chi_B'' u_1.
\end{equation*}
Moreover, since $SUL=L_0SU$, it holds
\begin{align*}
-(1-\gamma \partial_x^2)^{-1} SUL(\chi_B u_1) &=-(1-\gamma \partial_x^2)^{-1}L_0SU (\chi_B u_1)\\
&=-(1-\gamma \partial_x^2)^{-1}L_0[(1-\gamma \partial_x^2) v_1]\\
&= \partial_x^2 v_1-v_1 -\frac{\alpha-1}{\alpha+1} (1-\gamma \partial_x^2)^{-1} \left[ Q^{2\alpha} (1-\gamma \partial_x^2)v_1\right].
\end{align*}
Since
\begin{equation*}
(1-\gamma \partial_x^2)\left[Q^{2\alpha}v_1\right]
= Q^{2\alpha} (1-\gamma \partial_x^2) v_1 - 2\gamma (Q^{2\alpha})' \partial_x v_1
-\gamma (Q^{2\alpha})'' v_1,
\end{equation*}
we obtain
\begin{equation*}
-(1-\gamma \partial_x^2)^{-1} SUL(\chi_B u_1)
= -L_0 v_1 -\frac{\alpha-1}{\alpha+1} \gamma (1-\gamma \partial_x^2)^{-1} \left[ 2 (Q^{2\alpha})' \partial_x v_1
+ (Q^{2\alpha})'' v_1\right].
\end{equation*}
Therefore, we have obtained the following system for $(v_1,v_2)$
\begin{equation}\label{eq:v}\left\{
\begin{aligned}
 \dot v_1 & = v_2 \\
 \dot v_2 & =
 -L_0 v_1 
-\frac{\alpha-1}{\alpha+1} \gamma (1-\gamma \partial_x^2)^{-1} \left[ 2 (Q^{2\alpha})' \partial_x v_1
+ (Q^{2\alpha})'' v_1\right]\\
&\quad -(1-\gamma \partial_x^2)^{-1}SU\left[ 2 \chi_B' \partial_x u_1 + \chi_B''u_1\right]
+ (1-\gamma \partial_x^2)^{-1} SU [\chi_B N^\perp ].
\end{aligned}
\right.\end{equation}
For this transformed system we construct a second virial functional,
where the spectral analysis reduces to the fact that the potential in $L_0$ is positive.

\subsection{Virial functional for the transformed problem}
We set
\begin{equation*}
\mathcal J= \int \left(\psi_B\partial_x v_1+\frac 12 \psi_B' v_1\right)v_2,
\end{equation*}
and (see~\eqref{def:zetaB} and~\eqref{def:chiB})
\begin{equation}\label{def:z}
z=\chi_B \zeta_B v_1.
\end{equation}
Here, $z$ represents a localized version of the function $v_1$.
The scale of localization $B$ is intermediate between the one involved in the definition of $w$ from $u_1$ (see~\eqref{eq:A_B} and~\eqref{def:w}) and the weight function 
$\rho$ defined in~\eqref{def:rho} (similar to a localization at the soliton scale).

\begin{proposition}\label{virielJ}
There exist $C_2>0$ and $\delta_2>0$ such that for $\gamma$ small enough and 
for any $0<\delta\leq \delta_2$, the following holds. Fix $B=\delta^{-\frac 14}$.
Assume that for all $t\geq 0$,~\eqref{small} holds.
Then, for all $t\geq 0$,
\begin{equation}\label{eq:J}
\dot{\mathcal J} \leq - C_2 \|z\|_\loc^2 + \delta^{\frac 18} \|w\|_\loc^2 + |a_1|^3.
\end{equation}
\end{proposition}

\begin{remark}
The objective of estimate~\eqref{eq:J} is to control the local norm $ \|z\|_\loc^2$ up to small error in terms of $ \|w\|_\loc^2$ and $|a_1|^3$.
\end{remark}

The rest of this section is devoted to the proof of Proposition~\ref{virielJ}.
As in the computation of $\dot{\mathcal I}$ in the proof of Proposition~\ref{pr:virialu},
we have from~\eqref{eq:v},
\begin{equation*}
\begin{aligned}
\dot{\mathcal J} = & \int \left(\psi_B\partial_x v_1+\frac 12 \psi_B' v_1\right) \dot v_2 \\
= & -\int \left(\psi_B\partial_x v_1+\frac 12 \psi_B' v_1\right) L_0 v_1\\
& -\frac{\alpha-1}{\alpha+1} \gamma \int \left(\psi_B\partial_x v_1+\frac 12 \psi_B' v_1\right) (1-\gamma \partial_x^2)^{-1} \left[ 2 (Q^{2\alpha})' \partial_x v_1
+ (Q^{2\alpha})'' v_1\right] \\
& - \int \left(\psi_B\partial_x v_1+\frac 12 \psi_B' v_1\right) (1-\gamma \partial_x^2)^{-1}SU\left[ 2 \chi_B' \partial_x u_1 + \chi_B''u_1\right]\\
& + \int \left(\psi_B\partial_x v_1+\frac 12 \psi_B' v_1\right) (1-\gamma \partial_x^2)^{-1} SU [\chi_B N^\perp ]
=J_1+J_2+J_3+J_4.
\end{aligned}
\end{equation*}
First, using the definition of $L_0$ in \eqref{op:L0} and integrating by parts, we have
\begin{equation*}
J_1= - \int \psi_B' (\partial_x v_1)^2 +\frac14\int \psi_B''' v_1^2 -\frac{\alpha-1}{\alpha+1}\int \left(\psi_B\partial_x v_1+\frac 12 \psi_B' v_1\right) Q^{2\alpha}v_1.
\end{equation*}
From~\eqref{def:chiB}, we note that $\psi_B'= \chi_B^2 \zeta_B^2+(\chi_B^2)' \varphi_B$ and
\begin{equation*}
\psi_B''' = \chi_B^2 (\zeta_B^2)''+ 3 (\chi_B^2)' (\zeta_B^2)' + 3(\chi_B^2)'' \zeta_B^2 +(\chi_B^2)''' \varphi_B .
\end{equation*}
Thus,
\begin{align*}
\int \psi_B' (\partial_x v_1)^2 -\frac14 \int \psi_B''' v_1^2 
=& \int \chi_B^2 \zeta_B^2 (\partial_x v_1)^2 -\frac14 \int \chi_B^2 (\zeta_B^2)'' v_1^2 \\
& -\frac34 \int (\chi_B^2)'(\zeta_B^2)' v_1^2 -\frac34 \int (\chi_B^2)'' \zeta_B^2 v_1^2 \\
& 
+ \int (\chi_B^2)' \varphi_B (\partial_x v_1)^2-\frac14 \int (\chi_B^2)''' \varphi_B v_1^2.
\end{align*}
By the definition of $z$ in \eqref{def:z}, proceeding as in the proof of \eqref{avoir} in Lemma~\ref{le:1}, we have
\begin{align*}
\int \chi_B^2 \zeta_B^2 (\partial_x v_1)^2
&= \int (\partial_x z)^2 +\int (\chi_B \zeta_B)'' \chi_B \zeta_Bv_1^2\\
&=\int (\partial_x z)^2 + \int \frac{\zeta_B''}{\zeta_B} z^2 
+\int \chi_B'' \chi_B\zeta_B^2 v_1^2
+\frac 12 \int (\chi_B^2)'(\zeta_B^2)' v_1^2,
\end{align*}
and
\begin{equation*}
\frac14\int \chi_B^2 (\zeta_B^2)'' v_1^2 = \frac12 \int \left( \frac{\zeta_B''}{\zeta_B} +\frac{\zeta_B'^2}{\zeta_B^2}\right)z^2.
\end{equation*}
Thus,
\begin{equation*}
-\int \psi_B' (\partial_x v_1)^2 +\frac14 \int \psi_B''' v_1^2 
=-\left\{\int (\partial_x z)^2 +\frac 12 \int \left( \frac{ \zeta_B''}{\zeta_B}- \frac{ (\zeta_B')^2}{\zeta_B^2} \right) z^2 \right\}
+\widetilde J_1,\end{equation*}
where we have set
\begin{align*} 
\widetilde J_1&= \frac14 \int (\chi_B^2)' (\zeta_B^2)' v_1^2 +\frac12 \int \left[3(\chi_B')^2 + \chi_B'' \chi_B\right]\zeta_B^2 v_1^2\\
&\quad - \int (\chi_B^2)' \varphi_B (\partial_x v_1)^2+\frac14 \int (\chi_B^2)''' \varphi_B v_1^2.\end{align*}
Recalling~\eqref{def:z},~\eqref{def:chiB},~\eqref{def:zetaB} and integrating by parts,
\begin{equation*}
\int \left(\psi_B\partial_x v_1+\frac 12 \psi_B' v_1\right) Q^{2\alpha}v_1 = 
\frac 12 \int Q^{2\alpha} \partial_x \left(\psi_Bv_1^2\right) 
 = -\alpha \int \frac{\varphi_B}{ \zeta_B^2 } Q^{2\alpha-1}Q' z^2.
\end{equation*}

Therefore, setting 
\begin{equation*}
V=\frac 12 \left( \frac{ \zeta_B''}{\zeta_B}- \frac{ (\zeta_B')^2}{\zeta_B^2} \right) 
-\alpha\frac{\alpha-1}{\alpha+1} \frac{\varphi_B}{ \zeta_B^2 } Q^{2\alpha-1}Q' ,
\end{equation*}
we have obtained
\begin{equation*}
J_1
 =- \int \left[ (\partial_x z)^2 +V z^2\right]+\widetilde J_1.
\end{equation*}
\begin{lemma}
There exists $B_0>0$ such that for all $B\geq B_0$, $V\geq 0$ on $\RR$.
More precisely,
\begin{equation}\label{on:VB}
V \geq V_0 \quad \hbox{where} \quad V_0=\frac \alpha2\frac{\alpha-1}{\alpha+1} |x Q' | Q^{2\alpha-1}\geq 0.
\end{equation}
\end{lemma}
\begin{proof}
First, from~\eqref{bd:A} (with $A$ replaced by $B$), it holds
\begin{equation*}
\left|\frac{\zeta_B''}{\zeta_B}-\frac{(\zeta_B')^2}{\zeta_B^2}\right|\lesssim \frac {\mathbf{1}_{1\leq |x|\leq 2} (x)}B.\end{equation*}
Second, since for $x\in [0,+\infty)\mapsto \zeta_B(x)$ is non-increasing, we have for $x\geq 0$,
\begin{equation*}
\frac{\varphi_B}{\zeta_B^2}=\frac{\int_0^x \zeta_B^2}{\zeta_B^2}\geq x.
\end{equation*}
Since $Q'(x)\leq 0$ for $x\geq 0$, we obtain, for a constant $C>0$,
\begin{align*}
V(x)&\geq - \frac {C}B\mathbf{1}_{1\leq |x|\leq 2} (x)+\alpha\frac{\alpha-1}{\alpha+1} |x Q'(x)| Q^{2\alpha-1}(x)
\\
&\geq \frac \alpha2\frac{\alpha-1}{\alpha+1} |x Q'(x)| Q^{2\alpha-1}(x),
\end{align*}
choosing $B_0$ large enough. By parity, this estimate holds for any $x\in \RR$.
\end{proof}

Using this lemma, and the above computations for $J_1$,
we conclude
\begin{equation}\label{eq:J0}
\dot{\mathcal J} \leq 
- \int \left[ (\partial_x z)^2 + V_0 z^2 \right] +\widetilde J_1+J_2+J_3+J_4.
\end{equation}
To control the terms $\widetilde J_1$, $J_2$, $J_3$ and $J_4$, we need some technical estimates.

\subsection{Technical estimates} 
\begin{lemma}
\begin{enumerate}
\item Estimates on $w$.
\begin{equation}\label{eq:wun}\begin{aligned}
\int_{|x|\leq 2B^2} w^2 
&\lesssim
B^4 \int (\partial_x w)^2 + B^2 \int w^2\sech\left(\frac x2\right),
\end{aligned}\end{equation}
\begin{equation}\label{eq:wdeux} 
\|w\|_\loc^2 
 \lesssim \int (\partial_x w)^2 + \int_{|x|<1} w^2 
 \lesssim \int (\partial_x w)^2 + \int w^2\sech\left(\frac x2\right).
\end{equation}
\item Estimates on $z$.
\begin{equation}\label{eq:zun}
\|z\|_\loc^2 \lesssim
 \int (\partial_x z)^2 + \int V_0 z^2 \lesssim \|z\|_\loc^2 ,
\end{equation}
\begin{equation}\label{eq:zdeux}
\int z^2 \zeta_B \lesssim
 B^2 \int (\partial_x z)^2 + B\int V_0 z^2 \lesssim B^2\|z\|_\loc^2.
\end{equation}
\item Estimates on $v_1$.
\begin{equation}\label{eq:vun}
\|v_1\|_{L^2}\lesssim \gamma^{-1} B^2\|w\|_\loc,
\end{equation}
\begin{equation}\label{eq:vdeux}
\|\partial_x v_1\|_{L^2} \lesssim \gamma^{-1} \|w\|_\loc.
\end{equation}
\end{enumerate}
\end{lemma}
\begin{proof}
\emph{Proof of~\eqref{eq:wun} and~\eqref{eq:wdeux}.} For any $x,y\in \RR$, 
using $w(x)=w(y)+\int_y^x \partial_x w$ and the inequality
$(a+b)^2\leq 2a^2+2b^2$, we have
\begin{equation}\label{eq:pourw}\begin{aligned}
w^2(x)
&\leq 2 w^2(y)+2\left(\int_y^x\partial_x w\right)^2
\leq 2 w^2(y)+2|x-y| \int (\partial_x w)^2 \\
&\leq 2 w^2(y)+2(|x|+|y|) \int (\partial_x w)^2 .
\end{aligned}\end{equation}
Integrating~\eqref{eq:pourw} in $x\in [-2B^2,2B^2]$ and $y\in [-1,1]$, we find~\eqref{eq:wun}.
Multiplying~\eqref{eq:pourw} by $\sech\bigl(\frac{x}{10}\bigr)$ and integrating 
in $x\in \RR$ and $y\in [-1,1]$, we find~\eqref{eq:wdeux}.

\smallskip

\emph{Proof of~\eqref{eq:zun} and~\eqref{eq:zdeux}.} The proof is similar.
For any $x\in \RR$ and $y\in \RR$, we have
\begin{equation*}
z^2(x)
\leq 2 z^2(y)+2(|x|+|y|) \int (\partial_x z)^2 .
\end{equation*}
We multiply by $\sech\bigl(\frac{x}{10}\bigr)$ and $V_0(y)\geq 0$ and
integrate in $x\in \RR$ and $y\in\RR$.
Since $\int V_0 >0$ and $\int |y| V_0(y) dy<\infty$ from~\eqref{on:VB}, we obtain~\eqref{eq:zun}.

We multiply by $\zeta_B(x)$ and $V_0(y)$ and
integrate in $x\in \RR$ and $y\in\RR$.
Since 
\begin{equation*}\int \zeta_B \lesssim B,\quad \int |x| \zeta_B \lesssim B^2\quad\mbox{and}\quad
 \int |y|V_0\lesssim 1,
\end{equation*}
 we obtain~\eqref{eq:zdeux}.
\smallskip

\emph{Proof of~\eqref{eq:vun} and~\eqref{eq:vdeux}.}
Note by direct computations that
\begin{align*}
SU f 
&= f''-\left[\frac{Q'}{Q}+\frac{Y_0'}{Y_0}\right] f'
+ \left[-\left(\frac{Y_0'}{Y_0}\right)'+\frac{Q'}Q\frac{Y_0'}{Y_0}\right] f\\
&= f'' + (\alpha+2) \tanh(\alpha x) f'+(\alpha+1)\left(1+\frac{\alpha-1}{\cosh^2(\alpha x)}\right) f.
\end{align*}
Thus,
\begin{equation*}
\|SU f\|_{L^2}\lesssim \|f\|_{H^2}.
\end{equation*}
Moreover, using Fourier analysis,
\begin{equation*}
\|(1-\gamma \partial_x^2)^{-1} f\|_{H^2} \lesssim \gamma^{-1} \|f\|_{L^2}.
\end{equation*}
As a consequence, it holds
\begin{equation}\label{eq:SUgamma}
\|(1-\gamma \partial_x^2)^{-1} SU f\|_{L^2} \lesssim \gamma^{-1} \|f\|_{L^2}.
\end{equation}
Using~\eqref{eq:SUgamma}, the definition of $v_1$ in~\eqref{def:v1v2}, the definition of $w$ in \eqref{def:w} and $A\gg B^2$, we obtain
\begin{equation*}
\|v_1\|_{L^2}\lesssim \gamma^{-1} \|\chi_B u_1\|_{L^2}
\lesssim \gamma^{-1} \| u_1\|_{L^2(|x|\leq2B^2)}
\lesssim \gamma^{-1} \|w\|_{L^2(|x|\leq2B^2)},
\end{equation*}
and then~\eqref{eq:wun} implies~\eqref{eq:vun}.

Moreover, by direct computation
\begin{equation*}
\partial_x (SU f) = SU f' + (\alpha+2)\alpha \sech^2(\alpha x) f'
+\alpha(\alpha^2-1)\sech^2(\alpha x)\tanh(\alpha x) f.
\end{equation*}
Thus, similarly,
\begin{equation}\label{eq:SUgammax}
\|\partial_x (1-\gamma \partial_x^2)^{-1} SU f\|_{L^2} \lesssim \gamma^{-1} \|f'\|_{L^2}
+\|f \sech(x)\|_{L^2}.
\end{equation}
Using~\eqref{eq:SUgammax}, we obtain
\begin{equation*}
\|\partial_x v_1\|_{L^2}\lesssim \gamma^{-1} \|\partial_x(\chi_B u_1)\|_{L^2}
+ \|\chi_B u_1 \sech(x)\|_{L^2}.
\end{equation*}
By the definition of $w$, $A\gg B^2$ and the definition of $\chi_B$ and $\zeta_A$, we have
\begin{align*}
\left|\partial_x(\chi_B u_1)\right|^2 =
\left|\partial_x\left(\frac{\chi_B}{\zeta_A} w\right)\right|^2
&\lesssim \left|\frac{\chi_B}{\zeta_A}\right|^2 |\partial_x w|^2
+\left| \left(\frac{\chi_B}{\zeta_A}\right)'\right|^2w^2\\
&\lesssim |\partial_x w|^2 + B^{-4} w^2 \mathbf{1}_{|x|\leq 2B^2},
\end{align*}
and $ \|\chi_B u_1 \sech(x)\|_{L^2} \lesssim \|w \sech(x)\|_{L^2}$.
Thus, estimate~\eqref{eq:wun} imply~\eqref{eq:vdeux}.
\end{proof}

\begin{lemma}\label{le:5}
For any $0<K\leq 1$ and $\gamma>0$ small enough, for any 
$f\in L^2$,
\begin{equation}\label{eq:poids}
\| \sech(Kx) (1-\gamma \partial_x^2)^{-1} f \|_{L^2}
\lesssim \| \sech(Kx) f \|_{L^2}.
\end{equation}
where the implicit constant is independent of $\gamma$ and $K$.
\end{lemma}
\begin{proof}
We set $g=\sech(Kx) (1-\gamma \partial_x^2)^{-1} f$ and $k=\sech(Kx) f$.
We have
\begin{align*}
\cosh(Kx) k&= (1-\gamma \partial_x^2) [\cosh(Kx) g] \\
& = \cosh(Kx) g -\gamma K^2 \cosh(Kx) g - 2\gamma K \sinh(Kx) g' -\gamma \cosh(Kx) g''.
\end{align*}
Thus,
\begin{align*}
k= \left[ (1-\gamma K^2) -\gamma \partial_x^2\right] g - 2\gamma K \tanh(Kx) g'.
\end{align*}
For $0<K\leq 1$ and $\gamma\leq \frac 12$, we apply the operator $ \left[ (1-\gamma K^2) -\gamma \partial_x^2\right]^{-1}$, to obtain
\begin{align*}
g = \left[ (1-\gamma K^2) -\gamma \partial_x^2\right]^{-1} k + 2\gamma K \left[ (1-\gamma K^2) -\gamma \partial_x^2\right]^{-1} \left[\tanh(Kx) g'\right].
\end{align*}
For $0<K\leq 1$ and $\gamma\leq \frac 12$, one has
\begin{align*}
&\|\left[ (1-\gamma K^2) -\gamma \partial_x^2\right]^{-1}\|_{\mathcal L(L^2,L^2)}\lesssim 1,
\\
&\|\left[ (1-\gamma K^2) -\gamma \partial_x^2\right]^{-1}\partial_x\|_{\mathcal L(L^2,L^2)}\lesssim \gamma^{-\frac 12}.
\end{align*}
Thus,
$\|\left[ (1-\gamma K^2) -\gamma \partial_x^2\right]^{-1} k \|_{L^2}\lesssim \|k\|_{L^2},
$
and
\begin{align*}
&\|\left[(1-\gamma K^2) -\gamma \partial_x^2\right]^{-1} \left[\tanh(Kx) g'\right]\|_{L^2}
\\&\quad \lesssim 
\|\left[(1-\gamma K^2) -\gamma \partial_x^2\right]^{-1} \partial_x \left[\tanh(Kx) g\right]\|_{L^2}\\
&\quad \qquad + \|\left[(1-\gamma K^2) -\gamma \partial_x^2\right]^{-1} \left[\sech^2(Kx) g\right]\|_{L^2}\lesssim \gamma^{-\frac 12}\|g\|_{L^2}.
\end{align*}
We deduce, for a constant $C$ independent of $\gamma$,
\begin{equation*}
\|g\|_{L^2} \leq C \|k\|_{L^2} + C \gamma^{\frac 12} \|g\|_{L^2},
\end{equation*}
which implies~\eqref{eq:poids} for $\gamma$ small enough.
\end{proof}

\subsection{Control of error terms} Now, we are in a position to control the error terms in~\eqref{eq:J0}.

\smallskip

\emph{Control of $\widetilde J_1$.}
By the definition of $\zeta_B$, it holds 
\begin{equation*}
\zeta_B(x)\lesssim e^{-\frac {|x|}{B}},\quad
|\zeta'_B(x)|\lesssim\frac 1B e^{-\frac {|x|}{B}}.
\end{equation*}
Thus, using the properties of $\chi$ in \eqref{on:chi}, we have
\begin{align*}
 \int \left[ |\chi_B''| \chi_B \zeta_B^2 +
 (\chi_B')^2\zeta_B^2 
+ |\chi_B ' \zeta_B'|\chi_B \zeta_B\right] v_1^2
\lesssim \int_{B^2\leq |x|\leq 2B^2} e^{-\frac{2|x|}{B}} v_1^2
\lesssim e^{-2B} \|v_1\|_{L^2}^2.
\end{align*}
Next, since $|\varphi_B|\lesssim B$ and $|(\chi_B^2)'|\lesssim B^{-2}$, $|(\chi_B^2)'''|\lesssim B^{-6}$,
we have
\[
\int |(\chi_B^2)' \varphi_B| (\partial_x v_1)^2 
\lesssim B^{-1} \|\partial_x v_1\|_{L^2}^2
\quad \mbox{and}\quad
\int |(\chi_B^2)''' \varphi_B| v_1^2\lesssim B^{-5} \|v_1\|_{L^2}^2.
\]
Using~\eqref{eq:vun}-\eqref{eq:vdeux}, we conclude for this term
\begin{equation}\label{eq:J2}
|\widetilde J_1|\lesssim \gamma^{-2}B^{-1}\|w\|_\loc^2.
\end{equation}

\smallskip

\emph{Control of $J_2$.}
By the Cauchy-Schwarz inequality,
\begin{equation*}
|J_2|\lesssim \gamma 
\left\| Q^{\alpha} (1-\gamma \partial_x^2)^{-1} \left(\psi_B\partial_x v_1+\frac 12 \psi_B' v_1\right) \right\|_{L^2}
\left( \|Q^{\alpha} v_1\|_{L^2}+\|Q^{\alpha} \partial_x v_1\|_{L^2}\right).
\end{equation*}
First, we estimate using~\eqref{eq:poids}
\begin{equation*}
\left\| Q (1-\gamma \partial_x^2)^{-1} \left(\psi_B\partial_x v_1\right) \right\|_{L^2}
\lesssim \left\| Q \psi_B\partial_x v_1 \right\|_{L^2}
\end{equation*}
From the definition of $z$ in \eqref{def:z}, we have
\begin{equation*}
\partial_x z = \zeta_B \chi_B \partial_x v_1
+ (\zeta_B \chi_B )' v_1,
\end{equation*}
and so
\begin{equation*}
\zeta_B^2 \chi_B^2 |\partial_x v_1|^2
\lesssim |\partial_x z|^2 + |(\zeta_B \chi_B)' v_1|^2.
\end{equation*}
Using $|\chi'| \lesssim 1$, the definitions of $\chi_B$ and $\zeta_B$ and again the definition of $z$
\begin{equation*}
|(\zeta_B \chi_B )' v_1|^2 \chi_B^2 \lesssim B^{-2} \zeta_B^2 \chi_B^2 v_1^2 \lesssim B^{-2} z^2 ,
\end{equation*}
and so
\begin{equation}\label{diese}
\zeta_B^2 \chi_B^4 |\partial_x v_1|^2
\lesssim |\partial_x z|^2\chi_B^2 + B^{-2} z^2
\lesssim |\partial_x z|^2+z^2.
\end{equation}
Thus, using $|\psi_B|\lesssim |x| \chi_B^2$, 
\begin{equation*}
|Q \psi_B\partial_x v_1 |^2 
\lesssim |x|^2 Q^{2} \chi_B^4 |\partial_x v_1 |^2 
\lesssim Q\zeta_B^2 \chi_B^4 |\partial_x v_1|^2\lesssim |\partial_x z|^2 + Q z^2 .
\end{equation*}
It follows that
\begin{equation*}
\left\| Q \psi_B\partial_x v_1 \right\|_{L^2} \lesssim \|z\|_\loc .
\end{equation*}

Second, we also estimate using~\eqref{eq:poids}
\begin{equation*}
\left\| Q (1-\gamma \partial_x^2)^{-1} \left( \psi_B' v_1\right) \right\|_{L^2}
\lesssim \left\| Q \psi_B' v_1 \right\|_{L^2}
\end{equation*}
We claim
\begin{equation}\label{on:psix}
(\psi_B')^2\lesssim \chi_B^2.
\end{equation}
Indeed, using $|\chi_B'|\lesssim B^{-2}$, $|\varphi_B|\lesssim |x|$, $\chi_B=0$ for $|x|\geq 2B^2$
and $\zeta_B\leq 1$,
\begin{equation*}
(\psi_B')^2 
\lesssim [\chi_B' \chi_B]^2\varphi_B^2 
+\zeta_B^4 \chi_B^4 \lesssim \chi_B^2.
\end{equation*}

Using~\eqref{on:psix}, we infer that $|(\psi_B')^2 v_1^2|\lesssim \chi_B^2 v_1^2$, thus
$|Q(\psi_B')^2 v_1^2|\lesssim z^2$,
 and so
\begin{equation*}
\| Q \psi_B' v_1 \|_{L^2} \lesssim \| Q^\frac 12 z \|_{L^2} 
\lesssim \|z\|_\loc .\end{equation*}

Now, we estimate $\|Q^{\alpha} v_1\|_{L^2}$ and $\|Q^{\alpha} \partial_x v_1\|_{L^2}$.
From the definition of $z$ in~\eqref{def:z}, we have
$e^{-|x|} v_1^2 \chi_B^2\lesssim z^2$.
Thus, from the definition of $\chi_B$,
\begin{equation*}
e^{-2|x|} v_1^2 \lesssim e^{-2|x|} v_1^2 \chi_B^2 + e^{-2B^2} v_1^2
\lesssim e^{-|x|} z^2 + e^{-2B^2} v_1^2.
\end{equation*}
It follows using also~\eqref{eq:vun} that
\begin{equation*}
\|e^{-|x|} v_1\|_{L^2}
\lesssim \|z\|_\loc 
+ e^{-\frac 12 B^2} \gamma^{-1} \|w\|_\loc.
\end{equation*}
Differentiating  $z=\chi_B\zeta_B v_1$, we have
\begin{equation*}
\chi_B \zeta_B \partial_x v_1 = \partial_x z- \frac{ \zeta_B'}{\zeta_B} z - \chi_B' \zeta_B v_1.
\end{equation*}
Thus, as before,
\begin{align*}
e^{-2|x|} (\partial_x v_1)^2
&\lesssim e^{-|x|} \left[(\partial_x z)^2 + z^2\right]+e^{-2B^2}\left[(\partial_x v_1)^2+v_1^2\right].
\end{align*}
It follows using~\eqref{eq:vun} and~\eqref{eq:vdeux} that
\begin{equation*}
\|e^{-|x|} \partial_x v_1\|_{L^2}
\lesssim \|z\|_\loc
+ e^{-\frac 12 B^2} \gamma^{-1} \|w\|_\loc.
\end{equation*}

Collecting these estimates, we conlude
\begin{equation}\label{eq:J3}
|J_2|\lesssim \gamma \|z\|_\loc^2 
+e^{-B} \|w\|_\loc \|z\|_\loc.
\end{equation}

\smallskip

\emph{Control of $J_3$.}
Using Cauchy-Schwarz inequality and~\eqref{eq:SUgamma}, we have
\begin{equation*}
|J_3|\lesssim \gamma^{-1} \left(\|\psi_B \partial_x v_1\|_{L^2} + \|\psi_B' v_1\|_{L^2} \right) 
\left(\|\chi_B' \partial_x u_1\|_{L^2}+\|\chi_B'' u_1\|_{L^2}\right).
\end{equation*}
First, using $|\psi_B|\lesssim B$ (from its definition and $|\varphi_B|\lesssim B$) and~\eqref{eq:vdeux},
\begin{equation*}
\|\psi_B \partial_x v_1\|_{L^2} \lesssim B \| \partial_x v_1\|_{L^2} 
\lesssim \gamma^{-1} B \|w\|_\loc .
\end{equation*}
Then, since $|\varphi_B|\lesssim B$ and $\varphi_B'=\zeta_B^2$, 
\begin{equation*}
\left|\psi_B'\right| = \left|2\chi'_B \chi_B \varphi_B+\zeta_B^2 \chi_B^2\right|
\lesssim B^{-1}+\zeta_B^2 \chi_B^2.
\end{equation*}
Thus, using the definition~\eqref{def:z}, $z=\chi_B \zeta_B v_1$ and then~\eqref{eq:vun},
\begin{equation*}
\|\psi_B' v_1\|_{L^2}^2 
\lesssim B^{-2} \|v_1\|_{L^2}^2 + \int \zeta_B^2 z^2
\lesssim \gamma^{-2} B^2 \|w\|_\loc^2 +B^2 \|z\|_\loc^2.
\end{equation*}
In conclusion,
\begin{equation}\label{7fev}
\|\psi_B \partial_x v_1\|_{L^2} + \|\psi_B' v_1\|_{L^2}
\lesssim 
\gamma^{-1} B \|w\|_\loc+B \|z\|_\loc.
\end{equation}

Second, differentiating $w=\zeta_A u_1$, we have $\partial_x w = \zeta_A' u_1 + \zeta_A \partial_x u_1$,
so that (using also the assumption $A\gg B^2$),
\begin{equation*}\mbox{for $|x|<A$,}\quad |\partial_x u_1|^2\lesssim A^{-2} |u_1|^2 + |\partial_x w|^2
\lesssim B^{-4}|w|^2 + |\partial_x w|^2.\end{equation*}
Thus, using also~\eqref{eq:wun},
\begin{align*}
\|\chi_B' \partial_x u_1\|_{L^2}^2
&\lesssim B^{-4} \int_{B^2 <|x|<2B^2} |\partial_x u_1|^2
\\&\lesssim B^{-4} \left[\int |\partial_x w|^2 +B^{-4} \int_{|x|<2B^2} |w|^2\right]
 \lesssim
B^{-4}\|w\|_\loc^2.
\end{align*}
Next, by the definition of $\chi_B$ and~\eqref{eq:wun},
\begin{equation*}
\|\chi_B'' u_1\|_{L^2}^2
 \lesssim B^{-8} \int_{B^2 <|x|<2B^2} |u_1|^2
\lesssim B^{-8} \int_{|x|<2B^2} |w|^2
 \lesssim B^{-4} \|w\|_\loc^2.
\end{equation*}
In conclusion,
\begin{equation}\label{7fev2}
 \|\chi_B' \partial_x u_1\|_{L^2}+\|\chi_B'' u_1\|_{L^2}
\lesssim B^{-2} \|w\|_\loc^2.
\end{equation}

Collecting~\eqref{7fev} and~\eqref{7fev2}, we obtain
\begin{equation}\label{eq:J4}
|J_3|\lesssim \gamma^{-2}B^{-1} \|w\|_\loc^2 + \gamma^{-1}B^{-1} \|w\|_\loc \|z\|_\loc.
\end{equation}

\smallskip

\emph{Control of $J_4$.}
Using the Cauchy-Schwarz inequality,~\eqref{eq:SUgamma} and then 
$N^\perp=N-N_0 Y_0$, we have
\begin{align*}
|J_4|&\lesssim \gamma^{-1} \left(\|\psi_B \partial_x v_1\|_{L^2} + \|\psi_B' v_1\|_{L^2} \right) \| \chi_B N^\perp\|_{L^2}
\\&\lesssim \gamma^{-1} \left(\|\psi_B \partial_x v_1\|_{L^2} + \|\psi_B' v_1\|_{L^2} \right) 
\left( \| \chi_B N\|_{L^2} +|N_0|\right).
\end{align*}
By~\eqref{sur:N}, $|a_1|\lesssim 1$, $\|u_1\|_{L^\infty}\lesssim 1$, and decay properties of $Y_0$ and $Q$, we have
\begin{align*}
\| \chi_B N\|_{L^2}
&\lesssim a_1^2 + \|u_1\|_{L^\infty} \|Q \chi_B u_1\|_{L^2} +|a_1|^{2\alpha+1} +\|u_1\|_{L^\infty}^{2\alpha}\|\chi_B u_1\|_{L^2}
\\
&\lesssim a_1^2 + \|u_1\|_{L^\infty} \|\chi_B u_1\|_{L^2}.
\end{align*}
Using $\chi_B\lesssim \zeta_A$ (since $A\gg B^2$ in~\eqref{eq:A_B}) and~\eqref{eq:wun},
it holds
\begin{equation*}
\|\chi_B u_1\|_{L^2}^2
\lesssim \int_{|x|\leq 2B^2} w^2 \lesssim
B^4 \|w\|_\loc^2.\end{equation*}
Moreover, from~\eqref{sur:N0},
\[
|N_0|\lesssim a_1^2+\|u_1\|_{L^\infty} \|w\|_\loc.
\]

Therefore, using again~\eqref{7fev}, we obtain
\begin{equation}\label{eq:J5}
|J_4|\lesssim \gamma^{-2}B \left( \|w\|_\loc+ \|z\|_\loc \right)
\left(a_1^2 + B^2 \|u_1\|_{L^\infty} \|w\|_\loc\right).
\end{equation}

\subsection{End of proof of Proposition~\ref{virielJ}} From~\eqref{eq:J0},~\eqref{eq:zun},~\eqref{eq:J2},~\eqref{eq:J3},~\eqref{eq:J4} and~\eqref{eq:J5}, it
follows that there exist $C_2>0$ and $C>0$ such that 
\begin{align*}
\dot{\mathcal J} &\leq 
-4 C_2 \|z\|_\loc^2 
 + C \gamma^{-2} B^{-1} \|w\|_\loc^2 + C \gamma \|z\|_\loc^2
+C e^{-B} \|w\|_\loc \|z\|_\loc
\\&\quad + C\gamma^{-1} B^{-1} \|w\|_\loc \|z\|_\loc
 +C\gamma^{-2} B \left(\|w\|_\loc+\|z\|_\loc\right)\left(a_1^2+B^2\|u_1\|_{L^\infty}\|w\|_\loc\right).
\end{align*}
We fix $\gamma>0$ such that $C\gamma \leq 2C_2$ and also small enough to satisfy Lemma~\ref{le:5}.

The value of $\gamma$ being now fixed, we do not mention anymore dependency in $\gamma$.
Using standard inequalities and $B$ large enough, we obtain, for a possibly large constant $C>0$,
\begin{equation*}
\dot{\mathcal J} \leq 
-C_2\|z\|_\loc^2 
+C B^{-1} \|w\|_\loc^2 
 +C B^3 \left(a_1^2+B^2\|u_1\|_{L^\infty}\|w\|_\loc\right)^2.
\end{equation*}
Choosing (as specified in the statement of Proposition~\ref{virielJ})
\begin{equation*}
B=\delta^{-\frac 14},
\end{equation*}
and next using the assumption~\eqref{small},
we have
\begin{equation*}
B^3 (B^2\|u_1\|_{L^\infty}\|w\|_\loc)^2 \lesssim 
\delta^{-\frac 74}\|u_1\|_{L^\infty}^2 \|w\|_\loc^2
\lesssim \delta^{\frac 14} \|w\|_\loc^2.
\end{equation*}
Therefore, using again~\eqref{small}, for $\delta$ small enough (to absorb some constants), we obtain
\begin{equation*}
\dot{\mathcal J} \leq 
-C_2 \|z\|_\loc^2 
+C\delta^{\frac 14} \|w\|_\loc^2 + B^3 a_1^4
\leq -C_2\|z\|_\loc^2 + \delta^{\frac 18} \|w\|_\loc^2 + |a_1|^3.
\end{equation*}
This estimate completes the proof of Proposition~\ref{virielJ}.

\section{Coercivity and proof of Theorem~\ref{th:1}}\label{S:5}
In this section, the constant $\gamma$ is fixed as in Proposition~~\ref{virielJ}.
\subsection{Coercivity results}
\begin{lemma}
Let $B>2$.
Let $u$ and $v$ be Schwartz functions related by
\begin{equation}\label{eq:coer2}
v=(1-\gamma \partial_x^2)^{-1} SU (\chi_B u).
\end{equation}
Assume
\begin{equation}\label{eq:ortho}
\langle u,Y_0\rangle=\langle u,Q'\rangle=0.
\end{equation}
It holds
\begin{equation}\label{eq:coer}
\int (\chi_B u)^2 \sech\left(\frac{x}2\right) 
\lesssim \int \left[ (\partial_x v)^2 + v^2\right] \rho^2
+e^{-B} \int u^2 \sech\left( \frac{x}2\right).
\end{equation}
\end{lemma}
\begin{proof}
Using the expression of $S$ and $U$,
we rewrite~\eqref{eq:coer2} as
\begin{equation*}
v-\gamma \partial_x^2 v = Q\partial_x\left( \frac {Y_0}Q\partial_x\left(\frac {\chi_B u}{Y_0}\right)\right),
\end{equation*}
and thus
\begin{equation*}
\partial_x\left( \frac {Y_0}Q\partial_x\left(\frac {\chi_B u}{Y_0} \right)+ \gamma \frac{\partial_x v}{Q}\right)
=\frac 1 Q\left( v-\gamma \frac{Q'}{Q} \partial_x v \right). 
\end{equation*}
Integrating between $0$ and $x>0$, this yields, for some constant $a$,
\begin{equation*}
 \frac {Y_0}Q\partial_x\left(\frac {\chi_B u}{Y_0} \right)+ \gamma \frac{\partial_x v}{Q} 
 =a + \int_0^x \left[\frac 1 Q\left( v-\gamma \frac{Q'}{Q} \partial_x v \right)\right],
\end{equation*}
which rewrites as
\begin{equation*}
\partial_x\left(\frac {\chi_B u}{Y_0} \right)=
 a\frac Q{Y_0} - \gamma \frac{\partial_x v}{{Y_0}}+\frac Q{Y_0}\int_0^x \left[\frac 1 Q\left( v-\gamma \frac{Q'}{Q} \partial_x v \right)\right]
 .
\end{equation*}
Integrating on $[0,x]$, $x>0$, and multiplying by ${Y_0}$, it holds, for some constant $b$,
\begin{equation}\label{on:uutilde}
\chi_B u
= b {Y_0} + a {Y_0} \int_0^{x} \frac Q{Y_0} +\tilde u,
\end{equation}
where
\begin{equation*}
\tilde u= {Y_0}\int_0^x \left\{ - \gamma \frac{\partial_x v}{{Y_0}}+\frac Q{Y_0}\int_0^y \left[\frac 1 Q\left( v-\gamma \frac{Q'}{Q} \partial_x v \right)\right]
 \right\}.
\end{equation*}

Let us now estimate $\int \tilde u^2 \sech\left(\frac{x}2\right) $.
First, by the Cauchy-Schwarz inequality,
\begin{equation*}
{Y_0} \int_0^x \frac{|\partial_x v|}{{Y_0}}
\lesssim {Y_0} \left( \int (\partial_x v)^2 \rho^2\right)^{\frac 12}
\left( \int_0^x (\rho Y_0)^{-2} \right)^{\frac 12} 
\lesssim \rho^{-1} \left( \int (\partial_x v)^2 \rho^2\right)^{\frac 12}.
\end{equation*}
Second,
\begin{equation*}
\frac Q{Y_0} \int_0^y \frac {|v|} Q \lesssim
 \frac Q{Y_0} \left( \int v^2 \rho^2\right)^{\frac 12}
\left( \int_0^y (\rho Q)^{-2} \right)^{\frac 12} 
\lesssim (\rho Y_0)^{-1} \left( \int v^2 \rho^2\right)^{\frac 12}
\end{equation*}
Thus,
\begin{equation*}
{Y_0}\int_0^x \frac Q{Y_0} \int_0^y \frac {|v|} Q \lesssim
\left( \int v^2 \rho^2\right)^{\frac 12} {Y_0}\int_0^x (\rho Y_0)^{-1}
\lesssim \rho^{-1} \left( \int v^2 \rho^2\right)^{\frac 12}.
\end{equation*}
Third, since $\frac {|Q'|}Q\lesssim 1$, we obtain similarly,
\begin{equation*}
{Y_0}\int_0^x \frac Q{Y_0} \int_0^y \frac {|Q' \partial_x v|} {Q^2} 
\lesssim \rho^{-1} \left( \int (\partial_x v)^2 \rho^2\right)^{\frac 12}.
\end{equation*}

Collecting these estimates, we obtain, for all $x\geq 0$,
\begin{equation*}
\tilde u^2 \rho^2\lesssim \int \left[ (\partial_x v)^2 + v^2\right] \rho^2.
\end{equation*}
The same holds for $x\leq 0$, and thus
\begin{equation*}
\int \tilde u^2 \sech\left(\frac{x}2\right)\lesssim \int \left[ (\partial_x v)^2 + v^2\right] \rho^2.
\end{equation*}

To complete the proof, we estimate the constants $a$ and $b$ in \eqref{on:uutilde}. Using~\eqref{eq:ortho} and parity property, projecting~\eqref{on:uutilde} on
$Y_0$ yields
\begin{equation*}
\langle \chi_B u,Y_0\rangle
=\langle (\chi_B-1) u,Y_0\rangle
=b + \langle \tilde u,Y_0\rangle.
\end{equation*}
Thus,
\begin{equation*}
b^2\lesssim \int \tilde u^2 \sech\left(x\right) + 
\int u^2 \sech\left(x\right) (1-\chi_B)^2
\lesssim \int \tilde u^2 \sech\left(x\right) + e^{-\frac 12 B^2} \int u^2 \sech\left(\frac x2\right).
\end{equation*}
Using~\eqref{eq:ortho}, ${Y_0} \int_0^{x} \frac Q{Y_0}=-\alpha^{-1}Q'$ and projecting~\eqref{on:uutilde} on
$Q'$ yields similarly
\begin{equation*}
a^2\lesssim \int \tilde u^2 \sech\left(x\right) + e^{-\frac 12 B^2} \int u^2 \sech\left(\frac x2\right).
\end{equation*}
We conclude the proof using again~\eqref{on:uutilde}.
\end{proof}
The next result is a consequence of the previous general lemma, in the framework of the time-dependent functions introduced in~\eqref{myortho},
\eqref{def:w},~\eqref{def:v1v2} and~\eqref{def:z}.
\begin{lemma}
For $B$ large enough, it holds
\begin{equation}\label{eq:A}
\int w^2 \sech\left(\frac x2\right)\lesssim \|z\|_\loc^2 + e^{-B} \|\partial_x w\|_{L^2}^2,
\end{equation}
and
\begin{equation}\label{eq:B}
\|w\|_\loc^2 \lesssim \|z\|_\loc^2 + \|\partial_x w\|_{L^2}^2.
\end{equation}
\end{lemma}
\begin{proof}
Recall that the function $u_1$ is even so that it satisfies $\langle u_1,Q'\rangle=0$
in addition to the orthogonality~\eqref{myortho}. Therefore, applying~\eqref{eq:coer},
\begin{equation*}
\int (\chi_B u_1)^2 \sech\left(\frac{x}2\right) 
\lesssim \int \left[ (\partial_x v_1)^2 + v_1^2\right] \rho^2
+e^{-B} \int u_1^2 \sech\left( \frac{x}2\right),
\end{equation*}
which implies by~\eqref{def:w} and~\eqref{def:w_loc}
\begin{equation}\label{eq:rew}
\int (\chi_B w)^2 \sech\left(\frac{x}2\right) 
\lesssim \int \left[ (\partial_x v_1)^2 + v_1^2\right] \rho^2
+e^{-B} \|w\|_\loc^2.
\end{equation}
By~\eqref{def:z} and~\eqref{diese}, it holds
\begin{equation*}
\mbox{for $|x|<B^2$,}\quad 
\rho |\partial_x v_1|^2 +\rho |v_1|^2
\lesssim |\partial_x z|^2+z^2.
\end{equation*}
Thus, using~\eqref{eq:vun}-\eqref{eq:vdeux},
\begin{align*}
\int \left[ (\partial_x v_1)^2 + v_1^2\right] \rho^2
&\lesssim \int_{|x|<B^2} \left[ (\partial_x v_1)^2 + v_1^2\right] \rho^2
+e^{-\frac {B^2}{5}} \|v_1\|_{H^1}^2 \\
&\lesssim \|z\|_\loc^2 + e^{-\frac {B^2}{5}} \|v_1\|_{H^1}^2
\lesssim \|z\|_\loc^2 + e^{-\frac {B^2}{10}} \|w\|_\loc^2.
\end{align*}

Using~\eqref{eq:wdeux} and the definition of $\chi_B$ in~\eqref{def:chiB}, it holds
\begin{equation*}
\|w\|_\loc^2 
 \lesssim \int (\partial_x w)^2 + \int_{|x|<1} w^2 
 \lesssim \int (\partial_x w)^2 + \int (\chi_B w)^2 \sech\left(\frac{x}2\right) .
\end{equation*}
Inserting these estimates into~\eqref{eq:rew}, it follows for $B$ large enough that
\begin{equation*}
\int (\chi_B w)^2 \sech\left(\frac{x}2\right) 
\lesssim \|z\|_\loc^2
+e^{-B} \|\partial_x w\|_{L^2}^2.
\end{equation*}
The last two estimates imply~\eqref{eq:B}.

Finally, 
\begin{align*}
\int w^2 \sech\left(\frac{x}2\right)
&\lesssim \int (\chi_B w)^2 \sech\left(\frac{x}2\right) +e^{-\frac{B^2}4} \int w^2 \rho\\
&\lesssim \int (\chi_B w)^2 \sech\left(\frac{x}2\right) + e^{-B} \|w\|_\loc^2,
\end{align*}
and~\eqref{eq:A} follows.
\end{proof}

\subsection{Proof of Theorem~\ref{th:1}}
Recall that the constants $\gamma>0$, $\delta_1,\delta_2>0$ were defined in Propositions~\ref{pr:virialu} and~\ref{virielJ}.
\begin{proposition}
There exist $C_3>0$ and $0<\delta_3\leq \min(\delta_1,\delta_2)$ such that for any $0<\delta\leq \delta_3$, the following holds.
Fix $A=\delta^{-1}$ and $B=\delta^{-\frac 14}$. Assume that for all $t\geq 0$, \eqref{small} holds.

Let
\begin{equation}\label{def:H}
\mathcal H = \mathcal J + 8\delta_3^{\frac 1{10}} \mathcal I.
\end{equation}
Then, for all $t\geq 0$,
\begin{equation}\label{eq:prvirial}
\dot{\mathcal H}\leq -C_3 \|w\|_\loc^2 + 2|a_1|^3.
\end{equation}
\end{proposition}

\begin{proof}
In the context of Propositions~\ref{pr:virialu} and~\ref{virielJ},
observe that fixing $A=\delta^{-1}$ and $B=\delta^{-\frac 14}$, for $\delta>0$ small is consistent with the requirement $A\gg B^2\gg B\gg 1$ in~\eqref{eq:A_B}.
 
Combining~\eqref{eq:J} with~\eqref{eq:B} and~\eqref{eq:pr:u} with~\eqref{eq:A}, for $\delta_3>0$ small enough
and $0<\delta\leq \delta_3$, one obtains,
for a constant $C>0$,
\begin{align*}
\dot{\mathcal J}&\leq -\frac{C_2}2 \|z\|_\loc^2 + \delta_3^{\frac 1{10}}\|\partial_x w\|_{L^2}^2+|a_1|^3,\\
\dot{\mathcal I}&\leq -\frac 14 \|\partial_x w\|_{L^2}^2+C \|z\|_\loc^2 + |a_1|^3.
\end{align*}
Define $\mathcal H$ as in~\eqref{def:H}. It follows by combining the above estimates that
\begin{equation*}
\dot{\mathcal H}
\leq -\frac {C_2}2 \|z\|_\loc^2 - \delta_3^{\frac 1{10}} \|\partial_x w\|_{L^2}^2
+8C\delta_3^{\frac 1{10}} \|z\|_\loc^2 + \left(1+8 \delta_3^{\frac 1{10}}\right)|a_1|^3.
\end{equation*}
Possibly choosing a smaller $\delta_3$, we obtain
\begin{equation*}
\dot{\mathcal H}
\leq -\frac {C_2}4 \|z\|_\loc^2 - \delta_3^{\frac 1{10}} \|\partial_x w\|_{L^2}^2
 + 2|a_1|^3.
\end{equation*}
This estimate, together with~\eqref{eq:B}, implies~\eqref{eq:prvirial} for some $C_3>0$ (depending on $\delta_3$).
\end{proof}

We set
\begin{equation*}
\mathcal B = b_+^2-b_-^2.
\end{equation*}
\begin{lemma}\label{le:b}
There exist $C_4>0$ and $0<\delta_4\leq \delta_3$ such that for any $0<\delta\leq \delta_4$, the following holds.
Fix $A=\delta^{-1}$. Assume that for all $t\geq 0$, \eqref{small} holds. 
Then, for all $t\geq 0$,
\begin{equation}\label{eq:ble}
|\dot b_+-\nu_0 b_+|+|\dot b_-+\nu_0 b_-|\leq C_4\left( b_+^2 +b_-^2+ \|w\|_\loc^2\right),
\end{equation}
and
\begin{equation}\label{eq:bcle}
 \left|\frac{d}{dt} (b_+^2)-2\nu_0 b_+^2\right|
 +\left|\frac{d}{dt} (b_-^2)+2 \nu_0 b_-^2\right|
 \leq C_4\left( b_+^2 +b_-^2 + \|w\|_\loc^2\right)^{\frac 32}.
\end{equation}
In particular,
\begin{equation}\label{eq:BB}
\dot{\mathcal B}\geq \nu_0 \left( b_+^2+b_-^2\right) - C_4\|w\|_\loc^2
= \frac{\nu_0}2 \left( a_1^2+a_2^2\right) - C_4\|w\|_\loc^2.
\end{equation}
\end{lemma}
\begin{proof}
From~\eqref{sur:N0} and~\eqref{def:b}, it holds
\begin{equation*}
|N_0|\lesssim a_1^2+\|w\|_\loc^2
\lesssim b_+^2 +b_-^2 +\|w\|_\loc^2.
\end{equation*}
Estimates~\eqref{eq:ble} and~\eqref{eq:bcle} then follow from~\eqref{eq:a}.
Last, estimate~\eqref{eq:BB} is a consequence of~\eqref{eq:bcle} taking $\delta_4>0$ small enough.
\end{proof}

Combining~\eqref{eq:prvirial} and~\eqref{eq:BB}, it holds
\begin{equation*}\dot{\mathcal B}-2\frac{C_4}{C_3} \dot {\mathcal H}
\geq \frac{\nu_0}2 (a_1^2+a_2^2) + C_4\|w\|_\loc^2 -4\frac {C_4}{C_3} |a_1|^3,
\end{equation*}
and thus, for possibly smaller $\delta>0$,
\begin{equation}\label{lastvirial}
\dot{\mathcal B}-2\frac{C_4}{C_3} \dot{\mathcal H}
\geq \frac{\nu_0}4 (a_1^2+a_2^2) + C_4\|w\|_\loc^2.
\end{equation}
By the choice of $A=\delta^{-1}$, the bound
$|\varphi_A|\lesssim A$, and \eqref{small}, we have for all $t\geq 0$,
\[
|\mathcal I| \lesssim A \|u_1\|_{H^1}\|u_2\|_{L^2}\lesssim \delta .
\]
Similarly, using also~\eqref{eq:SUgamma}, it holds
\[
|\mathcal J|\lesssim B\|v_1\|_{H^1}\|v_2\|_{L^2}\lesssim \delta \quad \mbox{and thus}
\quad |\mathcal H|\lesssim \delta.
\]
Estimate $|\mathcal B|\lesssim \delta^2$ is also clear from \eqref{small}.

Therefore, integrating estimate~\eqref{lastvirial} on $[0,t]$ and passing to the limit as 
$t\to +\infty$, it follows that
\begin{equation*}
\int_0^\infty \left[a_1^2+a_2^2 + \|w\|_\loc^2\right] dt \lesssim \delta.
\end{equation*}
Since $\int [(\partial_x u_1)^2 + u_1^2] \sech(x) \lesssim \|w\|_\loc^2$, this implies
\begin{equation}\label{integral}
\int_0^\infty \left\{a_1^2+a_2^2 + \int \left[(\partial_x u_1)^2 + u_1^2\right] \sech(x)\right\} dt \lesssim \delta.
\end{equation}

Using \eqref{integral}, we conclude the proof of Theorem~\ref{th:1} as in Section~5.2 of~\cite{KMM}.
Let
\[
\mathcal K=\int u_1 u_2 \sech(x)\quad \mbox{and}\quad
\mathcal G = \frac 12 \int \left[(\partial_x u_1)^2+u_1^2+u_2^2\right] \sech(x).
\]
Using \eqref{eq:u}, we have
\begin{align*}
\dot{\mathcal K} 
&= \int \left[\dot u_1 u_2+u_1\dot u_2\right] \sech(x)\\
&= \int \left[u_2^2 + u_1 (-L u_1+N^\perp)\right] \sech(x)\\
&= \int \left[u_2^2 - (\partial_x u_1)^2 - u_1^2\right]\sech(x)
+\frac 12 \int u_1^2 \sech''(x)
\\
&\quad +\int \left[ f(Q+a_1 Y_0+u_1)-f(Q)-a_1 f'(Q)Y_0-N_0 Y_0\right] u_1 \sech(x).
\end{align*}
We check that
\[
\left|
\int \left[ f(Q+a_1 Y_0+u_1)-f(Q)-a_1 f'(Q)Y_0-N_0 Y_0\right] u_1 \sech(x)\right|
\lesssim a_1^2 + \int u_1^2 \sech(x).
\]
(See \eqref{sur:N}-\eqref{sur:N0} in the proof of Lemma~\ref{le:2}.)
In particular, it follows that
\[
\int u_2^2 \sech(x)
\leq \dot{\mathcal K} + Ca_1^2+ C \int \left[ (\partial_x u_1)^2 + u_1^2\right]\sech(x).
\]
Using the bound $|\mathcal K|\lesssim \delta^2$ and \eqref{integral}, we deduce
\begin{equation}\label{integral2}
\int_0^\infty \left[ a_1^2+a_2^2 + \mathcal G \right] dt \lesssim \delta.
\end{equation}

Similarly, we check that
\begin{align*}
\dot{\mathcal G} 
=& \int \left[(\partial_x \dot u_1)(\partial_x u_1)+\dot u_1u_1+\dot u_2 u_2\right] \sech(x)\\
=&\int \left[(\partial_x u_2)(\partial_x u_1)+u_2 u_1+ (-L u_1+N^\perp) u_2\right] \sech(x)\\
=&-\int (\partial_x u_1) u_2 \sech'(x) \\
& +\int \left[ f(Q+a_1 Y_0+u_1)-f(Q)-a_1 f'(Q)Y_0-N_0 Y_0\right] u_2 \sech(x),
\end{align*}
and so, as before
\begin{equation}\label{integral3}
|\dot {\mathcal G}|\lesssim a_1^2 +\mathcal G.
\end{equation}

By \eqref{integral2}, there exists an increasing sequence $t_n\to +\infty$ such that 
\[
\lim_{n\to \infty}\left[ a_1^2(t_n)+a_2^2(t_n) + \mathcal G(t_n)\right]=0.
\]
For $t\geq 0$, integrating \eqref{integral3} on $[t,t_n]$, and passing to the limit as $n\to \infty$, we obtain
\[
\mathcal G(t)\lesssim 
\int_t^{\infty} \left[a_1^2 + \mathcal G\right] dt.
\]
By \eqref{integral2}, we deduce that $\lim_{t\to\infty} \mathcal G(t)=0$.

Finally, by \eqref{eq:a} and \eqref{sur:N0}, we have
\[
\left|\frac{d}{dt} (a_1^2) \right|+\left|\frac{d}{dt} (a_2^2) \right|\lesssim a_1^2+a_2^2+ \int u_1^2 \sech(x),
\]
and so as before, by integration on $[t,t_n]$ and $n\to \infty$,
\[
a_1^2(t)+a_2^2(t)
\lesssim 
\int_t^{\infty} \left[a_1^2 +a_2^2+ \mathcal G\right] dt,
\]
which proves $\lim_{t\to \infty} |a_1(t)|+|a_2(t)|=0$.

By the decomposition \eqref{eq:decomp}, this clearly implies~\eqref{th:1:asymp}.
The proof of Theorem~\ref{th:1} is complete.

\section{Proof of Theorem~\ref{th:2}}

\subsection{Conservation of energy}
 Using~\eqref{def:L} and~\eqref{eq:Y0} and performing a standard computation, we expand the conservation of energy~\eqref{eq:energy} for a solution $(\phi,\partial_t \phi)$ written under the form~\eqref{eq:decomp}
with the orthogonality conditions~\eqref{myortho}, to obtain
\begin{align*}
&2 \left\{E(\phi,\partial_t\phi)-E(Q,0)\right\}\\&\quad =
 \int \left\{(\partial_t \phi)^2 + (\partial_x \phi)^2 + \phi^2 - 2 F(\phi)\right\} -2E(Q,0)\\
 &\quad =a_2^2\nu_0^2\langle Y_0,Y_0\rangle 
+a_1^2 \langle L Y_0,Y_0\rangle +\|u_2\|_{L^2}^2+ \langle L u_1, u_1\rangle+O\left(|a_1|^3 + |a_2|^3 +\|u_1\|_{H^1}^3\right)\\
&\quad = \nu_0^2 (a_2^2-a_1^2) +\|u_2\|_{L^2}^2+ \langle L u_1, u_1\rangle+O\left(|a_1|^3 + |a_2|^3+ \|u_1\|_{H^1}^3\right).
\end{align*}
Using the notation~\eqref{def:b}, we have
\begin{equation}\label{energlin}
\begin{aligned}
2 \left\{E(\phi,\partial_t\phi)-E(Q,0)\right\}
 = &-4\nu_0 b_+b_-+\|u_2\|_{L^2}^2+ \langle L u_1, u_1\rangle
\\& +O\left(|b_+|^3 + |b_-|^3 + \|u_1\|_{H^1}^3\right).
\end{aligned}
\end{equation}
Let $\delta_0>0$ be defined by
\begin{equation*}
\delta_0^2 = b_+^2(0)+b_-^2(0)+\|u_1(0)\|_{H^1}^2+\|u_2(0)\|_{L^2}^2.
\end{equation*}
Then,~\eqref{energlin} applied at $t=0$ gives
$|2 \left\{E(\phi,\partial_t\phi)-E(Q,0)\right\}|\lesssim \delta_0^2$.
Thus, by conservation of energy, estimate~\eqref{energlin} at some $t>0$ gives
\begin{equation*}
\left|-4 \nu_0 b_+b_-+\|u_2\|_{L^2}^2+ \langle L u_1, u_1\rangle+O\left(|b_+|^3 + |b_-|^3 + \|u_1\|_{H^1}^3\right)\right|
\lesssim \delta_0^2.
\end{equation*}
Under the orthogonality conditions~\eqref{myortho}, the parity of $u_1$, from the spectral analysis recalled in the Introduction
(see~\cite{CGNT}), it follows that for some $\mu>0$,
\begin{equation}\label{coercivity}
\langle L u_1,u_1\rangle \geq \mu \|u_1\|_{H^1}^2.
\end{equation}
Thus, as long as $\|u_1\|_{H^1}+\|u_2\|_{L^2}+|b_+|+|b_-|\lesssim \delta_0^{1/2}$, the following energy estimate holds
\begin{equation}\label{for:u}
\|u_1\|_{H^1}^2+\|u_2\|_{L^2}^2 \lesssim |b_+|^2+|b_-|^2 + \delta_0^2.
\end{equation}

\subsection{Construction of the graph}
By  the energy estimate~\eqref{for:u}, Lemma~\ref{le:b} and a standard contradiction argument, we construct initial data leading to global solutions close to the ground state $Q$.

\smallskip

Let
$\ee=(\varepsilon_1,\varepsilon_2)\in \mathcal A_{0}$ (see~\eqref{def:A_delta}).
Then,
the condition $\langle \ee,\ZZp\rangle=0$ rewrites
\begin{equation*}
\langle \varepsilon_1,Y_0\rangle+\langle \varepsilon_2,\nu_0^{-1} Y_0\rangle=0.
\end{equation*}
Define $b_-(0)$ and $(u_1(0),u_2(0))$ such that
\begin{equation*}
b_-(0)=\langle \varepsilon_1,Y_0\rangle = - \langle \varepsilon_2,\nu_0^{-1} Y_0\rangle
\end{equation*} 
and
\begin{equation*}
\varepsilon_1=b_-(0) Y_0 + u_1(0),\quad \varepsilon_2=- b_-(0)\nu_0 Y_0 + u_2(0).
\end{equation*}
Then, it holds
\begin{equation*}
\langle u_1(0),Y_0\rangle= \langle u_2(0),Y_0\rangle=0.
\end{equation*}
This means that the initial data in the statement of Theorem~\ref{th:2} decomposes as (see~\eqref{decompb})
\begin{equation*}
\pp_0=\pp(0) = (Q,0)+(u_1,u_2)(0)+b_-(0)\YYm + h(\ee) \YYp.
\end{equation*}

Now, we prove that there exists at least a choice of $h(\ee)=b_+(0)$ such that the corresponding solution $\pp$ is global
and satisfies~\eqref{prop:1:stab}.

Let $\delta_0>0$ small enough and $K>1$ large enough to be chosen.
We introduce the following bootstrap estimates
\begin{align}
&\|u_1\|_{H^1}\leq K^2 \delta_0\quad\hbox{and}\quad \|u_2\|_{L^2}\leq K^2 \delta_0,\label{bs1}\\
&|b_-|\leq K \delta_0,\label{bs2}\\
&|b_+|\leq K^5 \delta_0^2 .\label{bs3}
\end{align}
Given any $(u_1(0),u_2(0))$ and $b_-(0)$ such that 
\begin{equation}\label{initial}
\|u_1(0)\|_{H^1}\leq \delta_0,\quad \|u_2(0)\|_{L^2}\leq \delta_0,\quad |b_-(0)|\leq \delta_0,
\end{equation}
and $b_+(0)$ satisfying
\begin{equation*}
|b_+(0)|\leq K^5\delta_0^2,
\end{equation*}
we define
\begin{equation*}
T=\sup\{\mbox{$t\geq 0$ such that~\eqref{bs1}-\eqref{bs2}-\eqref{bs3} hold on $[0,t]$}\}.
\end{equation*}
Note that since $K>1$, $T$ is well-defined in $[0,+\infty]$.
We aim at proving that there exists at least one value of 
$b_+(0)\in [-K^5\delta_0^2,K^5\delta_0^2]$ such that $T=\infty$.
We argue by contradiction, assuming that any $b_+(0)\in [-K^5\delta_0^2,K^5\delta_0^2]$ leads to $T<\infty$.

\smallskip

First, we strictly improve the estimate on $(u_1,u_2)$ in~\eqref{bs1}.
Indeed, by~\eqref{for:u} and~\eqref{bs2}-\eqref{bs3}, it holds
\begin{equation*}
\|u_1\|_{H^1}^2+\|u_2\|_{L^2}^2 \leq C_5( K^{10} \delta_0^4+K^2\delta_0^2+\delta_0^2),
\end{equation*}
for some constant $C_5>0$. Thus, under the constraints
\begin{equation}\label{constr1}
C_5K^{10}\delta_0^2 \leq \frac 14 K^4 ,\quad
C_5K^2\leq \frac 14 K^4, \quad C_5\leq \frac 14 K^4,
\end{equation}
it holds $\|u_1\|_{H^1}^2+\|u_2\|_{L^2}^2 \leq \frac 34 K^4 \delta_0^2$, which strictly improves~\eqref{bs1}.

\smallskip

Second, we use~\eqref{eq:bcle} to control $b_-$. By~\eqref{bs1}-\eqref{bs2}-\eqref{bs3}, 
since $\|w\|_\loc\lesssim \|u_1\|_{H^1}$, it holds
\begin{equation*}\left|\frac{d}{dt}\left(e^{2\nu_0 t} b_-^2\right)\right|\leq C_6\left(K^{15}\delta_0^6+K^6\delta_0^3\right)e^{2\nu_0 t},\end{equation*}
for some constant $C_6>0$. Thus, by integration on $[0,t]$ and using~\eqref{initial}, we obtain
\begin{equation*}
b_-^2\leq \frac{C_6}{2\nu_0}\left(K^{15}\delta_0^6+K^6\delta_0^3\right) + \delta_0^2.
\end{equation*}
Under the constaints
\begin{equation}\label{constr2}
\frac{C_6}{2\nu_0} K^{15}\delta_0^4\leq \frac 14 K^2,\quad
C_6 K^6\delta_0\leq \frac 14K^2,\quad
1\leq \frac 14 K^2,
\end{equation}
it holds $b_-^2\leq \frac 34 K^2 \delta_0^2$ which strictly improves~\eqref{bs2}.
\smallskip

By the previous estimates (under the constraints~\eqref{constr1}-\eqref{constr2})
and a continuity argument,
we see that if $T<+\infty$, then $|b_+(T)|=K^5\delta_0^2$.

\smallskip

Third, we observe that if $t\in [0,T]$ is such that $|b_+(t)|=K^5\delta_0^2$, then it follows
from~\eqref{eq:ble} that
\begin{align*}
\frac{d}{dt}(b_+^2)
&\geq 2\nu_0 b_+^2 -2C_4|b_+| (b_+^2+b_-^2+\|w\|_\loc^2)\\
&\geq 2\nu_0K^{10}\delta_0^4 - C_7K^{5}\delta_0^2\left(K^{10}\delta_0^4+K^4\delta_0^2\right),
\end{align*}
for some constant $C_7>0$.
Under the constraints
\begin{equation}\label{constr3}
C_7K^{15}\delta_0^2\leq \frac 12 \nu_0K^{10},\quad
C_7K^{9} \leq \frac 12 \nu_0K^{10},
\end{equation}
the inequality
\begin{equation*}
\frac{d}{dt} (b_+^2)\geq \nu_0 K^{10}\delta_0^4>0,
\end{equation*}
holds.
By standard arguments, such transversality condition implies that $T$ is the first time for which $|b_+(t)|=K^5\delta_0^2$
and moreover that $T$ is continuous in the variable $b_+(0)$
(see e.g. \cite{CMM,CM} for a similar argument).
Now, the image of the continuous map
\begin{equation*}
b_+(0)\in [-K^5\delta_0^2,K^5\delta_0^2]\mapsto b_+(T)\in \{-K^5\delta_0^2,K^5\delta_0^2\}
\end{equation*}
is exactly $\{-K^5\delta_0^2,K^5\delta_0^2\}$ (since the image of $-K^5\delta_0^2$ is $-K^5\delta_0^2$
and the image of $K^5\delta_0^2$ is $K^5\delta_0^2$), which is a contradiction.

As a consequence, provided the constraints in~\eqref{constr1}-\eqref{constr2}-\eqref{constr3} are all fullfilled, 
there exists at least one value of $b_+(0)\in (-K^5\delta_0^2,K^5\delta_0^2)$ such that $T=\infty$.

Finally, we easily see that to satisfy~\eqref{constr1}-\eqref{constr2}-\eqref{constr3}, it is sufficient first to fix
$K>0$ large enough, depending only on $C_5$, $C_6$ and $C_7$, and then to choose $\delta_0>0$ small enough.

\subsection{Uniqueness and Lipschitz regularity}
The following proposition implies both the uniqueness of the choice of $h(\ee)=b_+(0)$, for a given $\ee\in \mathcal A_0$, and the Lipschitz regularity of the graph $\mathcal M$ defined from the resulting map $\ee\in \mathcal A_0\mapsto h(\ee)$.
It is thus sufficient to complete the proof of Theorem~\ref{th:2}.
\begin{proposition}
There exist $C, \delta>0$ such 
if $\pp$ and $\tilde \pp$ are two even solutions of~\eqref{nlkg} satisfying
 \begin{equation}\label{preuve:stab}
\mbox{for all $t\geq 0$,}\quad
 \|\pp(t)-(Q,0)\|_{H^1 \times L^2 }< \delta,\quad
 \|\tilde \pp(t)-(Q,0)\|_{H^1 \times L^2 }< \delta
\end{equation}
then, decomposing
\begin{equation*}
\pp(0)=(Q,0)+\ee+b_+(0) \YYp,\quad
\tilde\pp(0)=(Q,0)+\tilde\ee+\tilde b_+(0) \YYp 
\end{equation*}
with $\langle \ee, \ZZp\rangle = \langle \tilde \ee,\ZZp\rangle=0$, it holds
\begin{equation}\label{dolla}
|b_+(0)-\tilde b_+(0)|\leq C \delta^{\frac 12} \|\ee-\tilde \ee\|_{H^1\times L^2}.
\end{equation}
\end{proposition}

\begin{proof}
We use the decomposition and the notation of Section~\ref{S:2} for the two solutions $\pp$ and $\tilde \pp$
satisfying~\eqref{preuve:stab}.
In particular, from \eqref{small}, there exists $C_0>0$ such that for all $t\geq 0$,
\begin{equation}\label{estchec}
\|u_1(t)\|_{H^1}+\|\tilde u_1(t)\|_{H^1}+\|u_2(t)\|_{L^2}+\|\tilde u_2(t)\|_{L^2}
+|b_\pm(t)|+\tilde b_\pm(t)|\leq C_0 \delta.
\end{equation}
We denote
\begin{align*}
&\check a_1=a_1-\tilde a_1,\quad \check a_2=a_2-\tilde a_2, \quad
\check b_+=b_+-\tilde b_+,\quad \check b_-=b_--\tilde b_-,\\
&\check u_1=u_1-\tilde u_1, \quad \check u_2=u_2-\tilde u_2, \quad
\check N=N-\tilde N,\quad \check N^\perp=N^\perp-\tilde N^\perp,
\quad \check N_0=N_0-\tilde N_0.
\end{align*}
Then, from \eqref{eq:a}, \eqref{eq:u}, the equations of $(\check u_1,\check u_2,\check b_+,\check b_-)$ write
\begin{equation}\label{eq:ccb}\left\{
\begin{aligned}
& \dot {\check b}_+ = \nu_0 \check b_+ + \frac{\check N_0}{2\nu_0} \\
& \dot {\check b}_- = -\nu_0 \check b_- - \frac{\check N_0}{2\nu_0}
\end{aligned}
\right.\qquad \mbox{and}\qquad\left\{
\begin{aligned}
& \dot {\check u}_1 = \check u_2 \\
& \dot {\check u}_2 =- L \check u_1 + \check N^\perp.
\end{aligned}
\right.\end{equation}
We claim that
\begin{equation}\label{taylor}
|\check N_0|+\|\check N^\perp\|_{L^2}\leq C \delta \left( |\check b_+|+|\check b_-|+\|\check u_1\|_{H^1}\right).
\end{equation}
Indeed, by Taylor formula, for any $v,\tilde v$, it holds (recall that $\alpha>1$)
\begin{align*}
&\left| f(Q+v)-f(Q)-f'(Q)v - \left[ f(Q+\tilde v)-f(Q)-f'(Q)\tilde v\right]\right|\\
&\quad \lesssim |v-\tilde v| \left(|v|+|\tilde v|\right)\left( Q^{2\alpha-1}+|v|^{2\alpha-1}+|\tilde v|^{2\alpha-1}\right)\lesssim |v-\tilde v| \left(|v|+|\tilde v|\right).
\end{align*}
Using this inequality for $\check N= N-\tilde N$, where $N$ is defined in~\eqref{def:N}, and~\eqref{estchec}, we obtain
\begin{equation*}
|\check N|\lesssim \left(|\check a_1| Y_0 + |\check u_1|\right)
\left( Y_0|a_1|+ Y_0 |\tilde a_1|+|u_1|+|\tilde u_1|\right).
\end{equation*}
Using the Cauchy-Schwarz inequality and again~\eqref{estchec}, we find
$\|\check N\|_{L^2} \lesssim \delta (|\check a_1| + |\check u_1|)$
and estimate~\eqref{taylor} follows.

\smallskip

Let
\begin{equation*}
\beta_+=\check b_+^2,\quad \beta_-=\check b_-^2,\quad
\beta_c=\langle L \check u_1,\check u_1\rangle+\langle \check u_2,\check u_2\rangle.
\end{equation*}
By~\eqref{eq:ccb} and~\eqref{taylor} (and the coercivity property~\eqref{coercivity} for $\check u_1$) we have, for some $K>0$,
\begin{equation}\label{ode}
 |\dot \beta_c |+ |\dot \beta_+ -2\nu_0 \beta_+ |+ |\dot \beta_- + 2\nu_0 \beta_- |
\leq K \delta \left(\beta_c+\beta_++\beta_-\right).
\end{equation}

For the sake of contradiction, assume that the following holds
\begin{equation}\label{contra}
0\leq K \delta\left(\beta_c(0)+\beta_+(0)+\beta_-(0)\right) < \frac {\nu_0}{10} \beta_+(0).
\end{equation}
We introduce the following boostrap estimate
\begin{equation}\label{BSS}
K \delta\left(\beta_c +\beta_+ +\beta_- \right) \leq \nu_0 \beta_+ .
\end{equation}
Define
\begin{equation*}
T =\sup\{ \mbox{$t>0$ such that~\eqref{BSS} holds}\}>0.
\end{equation*}
We work on the interval $[0,T]$.
Note that from~\eqref{ode} and~\eqref{BSS}, it holds
\begin{equation}\label{croiss}
\dot \beta_+ \geq 2\nu_0 \beta_+ - K \delta \left(\beta_c+\beta_++\beta_-\right)
\geq \nu_0 \beta_+.
\end{equation}
In particular, by standard arguments, $\beta_+$ is positive and increasing on $[0,T]$.

Next, by~\eqref{ode} and~\eqref{BSS},
\begin{equation*}
\dot \beta_c \leq \nu_0 \beta_+ \leq \dot \beta_+ 
\end{equation*}
and thus, by integration,
\begin{equation*}
\beta_c(t)\leq \beta_c(0)+\beta_+(t)-\beta_+(0)\leq \beta_c(0)+\beta_+(t).
\end{equation*}
Therefore, by \eqref{contra}, for $\delta$ small enough,
\begin{equation*}
K \delta \beta_c(t) \leq K \delta ( \beta_c(0)+\beta_+(t))
\leq \frac {\nu_0}{10} \beta_+(0) + K \delta \beta_+(t)
\leq \frac {\nu_0}5\beta_+(t).
\end{equation*}

Then, by~\eqref{ode} and~\eqref{BSS},
\begin{equation*}
\dot \beta_-\leq -2\nu_0 \beta_- +\nu_0 \beta_+,
\end{equation*}
and so by integration and~\eqref{contra},
\begin{equation*}
\beta_-(t) \leq e^{-2\nu_0t} \beta_-(0)+\nu_0 \beta_+(t) e^{-2\nu_0t} \int_0^t e^{2\nu_0 s} ds
\leq \beta_-(0)+\frac 12 \beta_+(t).
\end{equation*}
Therefore,
for $\delta$ small enough,
\begin{equation*}
K \delta \beta_-(t) \leq K \delta ( \beta_-(0)+\beta_+(t))
\leq \frac {\nu_0}{10} \beta_+(0) + K \delta \beta_+(t)
\leq \frac {\nu_0}5\beta_+(t).
\end{equation*}
Last, it is clear that for $\delta$ small, it holds $K \delta\beta_+\leq \frac {\nu_0}5\beta_+$.

Therefore, 
we have proved that, for all $t\in [0,T]$,
\begin{equation*}
K \delta\left(\beta_c(t)+\beta_+(t)+\beta_-(t)\right) \leq \frac 35 \nu_0 \beta_+(t).
\end{equation*}
By a continuity argument, this means that $T=+\infty$.
By the exponential growth~\eqref{croiss} and $\beta_+(0)>0$, we obtain a contradiction with the global bound~\eqref{estchec} on $|b_+|$.

Since estimate~\eqref{contra} is contradicted,
and since it holds
\begin{equation*}
\ee= \uu(0)+b_-(0)\YYm,\quad \tilde \ee= \tilde \uu(0)+\tilde b_-(0)\YYm
\quad \mbox{with}\quad \langle\uu(0),\YYm\rangle=\langle\tilde \uu(0),\YYm\rangle=0,
\end{equation*} 
we have proved \eqref{dolla}.
\end{proof}

\end{document}